\documentclass[]{interact}

\usepackage{epstopdf}% To incorporate .eps illustrations using PDFLaTeX, etc.
\usepackage[caption=false]{subfig}% Support for small, `sub' figures and tables

\usepackage[numbers,sort&compress]{natbib}% Citation support using natbib.sty
\bibpunct[, ]{[}{]}{,}{n}{,}{,}% Citation support using natbib.sty
\makeatletter% @ becomes a letter
\def\NAT@def@citea{\def\@citea{\NAT@separator}}% Suppress spaces between citations using natbib.sty
\makeatother% @ becomes a symbol again

\theoremstyle{plain}% Theorem-like structures provided by amsthm.sty
\newtheorem{theorem}{Theorem}[section]

\newtheorem{corollary}[theorem]{Corollary}

\theoremstyle{definition}
\newtheorem{definition}[theorem]{Definition}

\theoremstyle{remark}

\begin{document}

%\articletype{ELLIPTIC EQUATIONS}% Specify the article type or omit as appropriate

\title{Infinitely many sign-changing solutions for the nonlinear Schr\"odinger-Poisson system with $p$-laplacian}

\author{
\name{Shuo Ren, Huixing Zhang\thanks{CONTACT Huixing Zhang. Email: huixingzhangcumt@163.com}, Zhen Cheng and Yan Gao}
\affil{School of Mathematics, China University of Mining and Technology, Xuzhou, People's Republic of China}
}

\maketitle

\begin{abstract}
In this paper, we consider the following Schr\"odinger-Poisson system with $p$-laplacian
\begin{equation}
	\begin{cases}
		-\Delta_{p}u+V(x)|u|^{p-2}u+\phi|u|^{p-2}u=f(u)\qquad&x\in\mathbb{R}^{3},\\
		-\Delta\phi=|u|^{p}&x\in\mathbb{R}^{3}.
	\end{cases}\notag
\end{equation}
We investigate the existence of multiple sign-changing solutions. By using the method of invariant sets of descending flow, we prove that this system has infinitely many sign-changing solutions. This system is new one coupled by Schr\"odinger equation of $p$-laplacian with a Poisson equation. Our results complement the study made by Zhaoli Liu, Zhiqiang Wang, Jianjun Zhang (Annali di Matematica Pura ed Applicata, 195(3):775-794(2016)).
\end{abstract}


\section{Introduction}

In this paper, we are motivated by \cite{liu2016infinitely} and study the existence of sign-changing solutions to the following Schr\"odinger-Poisson system with $p$-laplacian
\begin{equation}
	\begin{cases}
		-\Delta_{p}u+V(x)|u|^{p-2}u+\phi|u|^{p-2}u=f(u)\qquad&x\in\mathbb{R}^{3},\\
		-\Delta\phi=|u|^{p}&x\in\mathbb{R}^{3},
	\end{cases}\label{1.1}\tag{1.1}
\end{equation}
where $\Delta_{p}u=\mathrm{div}(|\nabla u|^{p-2}\nabla u)$,\;$1<p<3$,\;$p^{*}=\frac{3p}{3-p}$.

From the view of mathematics, the system \eqref{1.1} can be regarded as an extension of the Schr\"odinger-Poisson system
\begin{equation}
	\begin{cases}
		-\Delta u+V(x)u+\phi u=f(u)\qquad &x\in\mathbb{R}^{3},\\
		-\Delta\phi=u^{2}&x\in\mathbb{R}^{3}.\label{1.2}\tag{1.2}
	\end{cases}
\end{equation}

System like \eqref{1.2} originates from quantum mechanics models \cite{benguria1981thomas,lieb1997thomas} and from semiconductor theory \cite{lions1987solutions,markowich2012semiconductor} and
describes the interaction of a quantum particle with an electromagnetic field. For the further physical background, we refer the readers to \cite{ambrosetti2008multiple,benci1998eigenvalue,benci2002solitary}.

The $p$-laplacian equation arises naturally in various contexts of physics, for instance, in the study of non-Newtonian fluids (the case of a Newtonian fluid corresponding to $p=2$), and in the study of nonlinear elasticity problems. It also appears in the search for solitons of certain Lorentz-invariant nonlinear field equations. There also have been many interesting works about the existence of positive solutions, multiple solutions, ground states and semiclassical states to $p$-laplacian type equation via variational methods, see for instance \cite{bartsch2004superlinear,bartsch2005nodal,du2021schrodinger,du2022quasilinear} and the references therein.

In recent years, there has been increasing attention to the existence of nontrivial solutions for the  $p$-laplacian equatuion. Bartsch, Liu and Weth \cite{liu2015multiple} study the quasilinear problem
\begin{equation}
	-\Delta_{p}u=f(x,u),\qquad u\in W_{0}^{1,p}(\Omega)\label{1.3}\tag{1.3}
\end{equation}
on a bounded domain $\Omega\subset\mathbb{R}^{N}$ with smooth boundary $\partial\Omega$. They set up
closed convex subsets of $X$ with non-empty interior and construct critical points outside of these closed convex subsets which will be nodal (that is, sign-changing) solutions of equation \eqref{1.3}. And they prove existence and multiplicity results about sign changing solutions of problem \eqref{1.3} with $f$ being subcritical and $f(x,u)/|t|^{p-2}$ being superlinear at $t=\infty$.

Liu, Zhao and Liu \cite{liu2018system} consider the system of $p$-laplacian equations with critical growth
\begin{equation}
	\begin{cases}
		-\Delta_{p} u_{j}=\mu_{j}|u_{j}|^{p-2}u_{j}+\sum_{i=1}^{k}\beta_{ij}|u_{i}|^{\frac{p^{*}}{2}}|u_{j}|^{\frac{p^{*}}{2}-2}\qquad&\text{in}\;\Omega,\\
		u_{j}=0&\text{on}\;\partial\Omega,\qquad j=1,2,\cdots,k,
	\end{cases}\notag
\end{equation}
where $\Omega$ is a bounded smooth domain in $\mathbb{R}^{N},1<p<N,p^{*}=\frac{Np}{N-p}>2,0<\mu_{j}<\lambda_{1}(\Omega)$ the first eigenvalue of the $p$-laplacian operator $-\Delta_{p}$ with the Dirichlet boundary condition, $\beta_{jj}>0,\beta_{ij}=\beta_{ji}\leqslant0$, for $i\not=j$. The existence of infinitely many sign-changing
solutions is proved by the truncation method and by the concentration analysis on the approximating solutions, provided $N>p+p^{2}$. Cao, Peng and Yan \cite{cao2012infinitely} prove the existence of infinitely many solutions for the following elliptic problem with critical Sobolev growth
\begin{equation}
	\begin{cases}
		-\Delta_{p}u=|u|^{p^{*}-2}u+\mu|u|^{p-2}u\qquad&\text{in}\;\Omega,\\
		u=0\qquad&\text{on}\;\partial\Omega,
	\end{cases}\notag
\end{equation}
provided $N>p^{2}+p,1<p<N,p^{*}=\frac{NP}{N-p},\mu>0$ and $\Omega$ is an open bounded domain in $\mathbb{R}^{N}$. Existence and multiplicity of solutions of $p$-laplacian equation have also gained much interest in recent years, see, for example, \cite{ambrosetti1996multiplicity,ambrosetti1997quasilinear,azorero1994some,chabrowski1995multiple,garcia1991multiplicity,pomponio2010schrodinger}.

There is an extensive work on the existence of solution of Schr\"odinger-Poisson system. Via the method of invariant sets of descending flow, Liu, Wang and Zhang prove that the following system has infinitely many sign-changing solutions
\begin{equation}
	\begin{cases}
		-\Delta u+V(x)u+\phi u=f(u)\qquad&\text{in}\;\mathbb{R}^{3},\\
		-\Delta\phi=u^{2}&\text{in}\;\mathbb{R}^{3}.
	\end{cases}\label{1.4}\tag{1.4}
\end{equation}
By minimax arguments in the presence of invariant sets, they obtain the existence of sign-changing solutions to \eqref{1.4}. For  infinitely many sign-changing solutions of \eqref{1.4}, the above framework is not directly applicable due to changes of geometric nature of the variational formulation. They use a perturbation approach by adding a term growing faster than the monomial of degree 4 with a small coefficient $\lambda>0$. For the perturbed problems, they apply the program above to establish the existence of multiple sign-changing solutions, and a convergence argument allows them to pass limit to the original system. By using the concentration compactness principle, Azzollini and Pomponio \cite{azzollini2008ground} prove the existence of a ground state solution of \eqref{1.4} when $f(u)=|u|^{p-2}u$ and $p\in(3,6)$. But no symmetry information concerning, this ground state solution was given. For more results of Schr\"odinger-Poisson system, we refer the readers to \cite{bartsch2004superlinear,bartsch2005nodal,bartsch1995existence,d2006standing,d2004non,d2002non,ianni2008concentration,liu2001invariant} and the references therein.

Du et al. \cite{du2022quasilinear} consider the following problem
\begin{equation}
	\begin{cases}
		-\Delta_{p}u+|u|^{p-2}u+\lambda\phi|u|^{p-2}u=|u|^{p-2}u\qquad&\text{in}\;\mathbb{R}^{3},\\
		-\Delta\phi=|u|^{p}&\text{in}\;\mathbb{R}^{3},
	\end{cases}\notag
\end{equation}
and produce the bounded Palais-Smale sequences. The existence of nontrivial solutions of the system is obtained by the mountain pass theorem. The key point is $q=2p$. For $2p\leqslant q<p^{*}$, the boundedness of the Palais-Smale sequence at the mountain pass level is easily proved. But for $p<q<2p$, it is tricky to obtain the boundedness of the Palais-Smale sequence. To overcome this difficulty, the scaling technique and the cut-off technique are used to deal with the case $p<q\leqslant \frac{2p(p+1)}{p+2}$ and the case $\frac{2p(p+1)}{p+2}<q<2p$, respectively.

In order to study the existence of sign-changing solutions to system \eqref{1.1}, we will adopt an idea from \cite{bartsch2005nodal} to construct an auxiliary operator $A$, which is the starting point in
constructing a pseudo-gradient vector field guaranteeing existence of the desired invariant sets of the flow. The method of invariant sets of descending flow has been used widely in dealing with sign-changing solutions of elliptic problem, see \cite{bartsch2004superlinear,bartsch2005nodal,liu2001invariant,liu2016infinitely} and the references therein. Moreover,  we need an abstract critical point theory developed by Liu et al. \cite{liu2015multiple}.

\subsection{Main results}\label{class}

In what follows, we assume $V(x)$ satisfies following conditions.\\
$(V_{0})\;V(x)\in C(\mathbb{R}^{3},\mathbb{R}),\,\inf_{\mathbb{R}^{3}}V(x)>0.$\\
$(V_{1})$\;There exists\;$r_{0},$ such that $M>0,$
\[\lim_{|y|\to\infty}m\{x\in \mathbb{R}^{3}:|x-y|\leqslant r_{0},V(x)\leqslant M\}=0,\]

where $m$ means Lebesgue measure on $\mathbb{R}^{3}$.\\
$(V_{2})\;V(x)$ is differentiable, for $\nabla V(x)\cdot x\in L^{\infty}(\mathbb{R}^{3})\cup L^{\frac{p^{*}}{p^{*}-p}}(\mathbb{R}^{3})$,
\[\frac{\mu-p}{2p-\mu}V(x)+\nabla V(x)\cdot x\geqslant0,\qquad a.e.\;x\in\mathbb{R}^{3}.\]
Moreover, we assume $f$ satisfies the following hypotheses.\\
$(f_{1})\;f\in C(\mathbb{R},\mathbb{R}),$ and $\lim_{s\to 0}\frac{f(s)}{|s|^{p-1}}=0.$\\
$(f_{2})$\;There exists\;$c>0$ and $p<q<p^{*}$ such that
\[|f(t)|\leqslant c(1+|t|^{q-1}),\qquad\forall t\in\mathbb{R}.\]
$(f_{3})$\;There exists\;$\mu>p,$ such that for $t\not=0$,
\[0<F(t):=\int_0^t f(s)ds\leqslant \frac{1}{\mu}tf(t).\]
Our main results on problem \eqref{1.1} are the following.

\newtheorem{thm}{Theorem}[section]
\begin{thm}
	If $(V_{0})-(V_{1})$ and $(f_{1})-(f_{3})$ hold and $\mu>2p$, then system \eqref{1.1} has one sign-changing solution. If moreover $f$ is odd, then problem \eqref{1.1} has infinitely many sign-changing solutions.\label{theorem1.1}
\end{thm}
\begin{thm}
	If $(V_{0})-(V_{2})$ and $(f_{1})-(f_{3})$ hold and $\mu>\frac{2p(p+1)}{p+2}$, then system \eqref{1.1} has one sign-changing solution. If in addition $f$ is odd, then problem \eqref{1.1} has infinitely many sign-changing solutions.\label{theorem1.2}
\end{thm}

The outline of our argument is as follows. In Section 2 we show the variational framework of our problem, some preliminary properties of $\phi_{u}$ and some lemmas. We prove some lemmas and Theorem \ref{theorem1.1} by an abstract critical point theory developed by Liu et al. \cite{liu2015multiple} in Section 3. Section 4 is devoted to proving Theorem \ref{theorem1.2} by a perturbation approach.

\section{Preliminaries and functional setting}

In this paper, we make use of the following notations.\\
\begin{itemize}
	\item $1\leqslant p<\infty,\|u\|_{p}=\bigl(\int_{\mathbb{R}^{3}}|u|^{p}\mathrm{d}x\bigr)^{\frac{1}{p}};$
	\item $\|u\|_{W^{1,p}}=\bigl(\int_{\mathbb{R}^{3}}(|\nabla u|^{p}+|u|^{p})\mathrm{d}x\bigr)^{\frac{1}{p}};$
	\item $D^{1,2}(\mathbb{R}^{3}),$ for $N\geqslant3,$ is the space defined as follows: let $2^{*}=\frac{2N}{N-2}$, then
	\[D^{1,2}(\mathbb{R}^{3})=\biggl\{u\in L^{2^{*}}(\mathbb{R}^{3})\;\bigg|\;\frac{\partial u}{\partial x_{i}}L^{2}(\mathbb{R}^{3}),i=1,2,\cdots,N\biggr\}.\]
	This space has a Hilbert structure when endowed with the scalar product
	\[(u,v)=\int_{\mathbb{R}^{3}}\nabla u\cdot\nabla v \mathrm{d}x,\]
	so that the corresponding norm is
	\[\|u\|=\biggl(\int_{\mathbb{R}^{3}}|\nabla u|^{2}\mathrm{d}x\biggr)^{\frac{1}{2}}.\]
	The space $C_{0}^{\infty}(\mathbb{R}^{3})$ is dense in $D^{1,2}(\mathbb{R}^{3})$.
	\item We define the Sobolev space
	\[E=\{u\in W^{1,p}(\mathbb{R}^{3}):\int_{\mathbb{R}^{3}}V(x)|u|^{p}\mathrm{d}x<\infty\},\]
	with the  corresponding norm
	\[\|u\|_{E}=\biggl(\int_{\mathbb{R}^{3}}\bigl(|\nabla u|^{p}+V(x)|u|^{p}\bigr)\mathrm{d}x\biggr)^{\frac{1}{p}}.\]
	Then $E$ is a Banach space.
	\item $C,C_{i}(i=1,2,\cdots,n):$ denote (possibly different) positive constants.
\end{itemize}

According to \cite{du2022quasilinear}, for any given $u\in W^{1,p}(\mathbb{R}^{3})$, there exists a unique  
\[\phi_{u}(x)=\frac{1}{4\pi}\int_{\mathbb{R}^{3}}\frac{|u(y)|^{p}}{|x-y|}\mathrm{d}y,\qquad \phi_{u}\in D^{1,2}(\mathbb{R}^{3})\]
such that 
\[-\nabla\phi_{u}=|u|^{p}.\]
We now summarize some properties of $\phi_{u}$, which will be used later. See, for instance, \cite{du2022quasilinear} for a proof.

\newtheorem{lem}{Lemma}[section]
\begin{lem}\label{lemma2.1}
	Let $u\in W^{1,p}(\mathbb{R}^{3})$, \\
	$(1)\ \phi_{u}\geqslant 0,x\in\mathbb{R}^{3};$\\
	$(2)$ for any $t\in \mathbb{R}^{+},\phi_{tu}=t^{p}\phi_{u}$, and $\phi_{u_{t}}t^{kp-2}\phi_{u}(tx)$ with $u_{t}(x)=t^{k}u(tx);$\\
	$(3)\;\|\phi_{u}\|_{D^{1,2}}\leqslant C\|u\|_{E}^{p}$, with $C$ independent of $u;$\\%\iff (\phi_{u}u^{p})^{\frac{1}{2}}\leqslant Cu^{p}%
	$(4)$\;if $u_{n}\rightharpoonup u$ in $W^{1,p}(\mathbb{R}^{3})$, then $\phi_{u_{n}}\rightharpoonup \phi_{u}$ in $D^{1,2}(\mathbb{R}^{3})$, and
	\[\int_{\mathbb{R}^{3}}\phi_{u_{n}}|u_{n}|^{p-2}u_{n}\varphi \mathrm{d}x\to \int_{\mathbb{R}^{3}}\phi_{u}|u|^{p-2}u\varphi \mathrm{d}x,\qquad \forall\;\varphi\in W^{1,p}(\mathbb{R}^{3}).\]
\end{lem}

\newtheorem{rem}{Remark}[section]
\begin{rem}
	\label{remark2.1} $D^{1,2}(\mathbb{R}^{3})\hookrightarrow L^{6}(\mathbb{R}^{3})$.
\end{rem}
\begin{rem}
	\label{remark2.2} By $(V_{0})-(V_{1})$, the embedding $E\hookrightarrow L^{q}(\mathbb{R}^{3}), q\in[\,p,p^{*})$ is compact. Similar to \cite{bartsch1995existence}.
\end{rem}

Referring \cite{liu2016infinitely}, we define
\[D\bigl(f,g\bigr)=\int_{\mathbb{R}^{3}}\int_{\mathbb{R}^{3}}\frac{f(x)g(y)}{4\pi|x-y|}\mathrm{d}x\mathrm{d}y.\]  Moreover, we have the following
property.
\begin{rem}\label{remark2.3}
	$|D\bigl(f,g\bigr)|^{2}\leqslant D\bigl(f,f\bigr)D\bigl(g,g\bigr)$, see \cite[p.250]{lieb2001analysis}.
\end{rem}

Substituting $\phi=\phi_{u}$ into system \eqref{1.1}, we can rewrite system \eqref{1.1} as the single equation
\begin{equation}
	-\Delta_{p}u+V(x)|u|^{p-2}u+\phi_{u}|u|^{p-2}u=f(u),\qquad u\in E.\label{2.1}\tag{2.1}
\end{equation}
We define the energy functional $I$ on $E$ by
\[I(u)=\frac{1}{p}\int_{\mathbb{R}^{3}}\bigl(|\nabla u|^{p}+V(x)|u|^{p}\bigr)\mathrm{d}x+\frac{1}{2p}\int_{\mathbb{R}^{3}}\phi_{u}|u|^{p}\mathrm{d}x-\int_{\mathbb{R}^{3}}F(u)\mathrm{d}x.\]
It is standard to show that $I\in C^{1}(E,\mathbb{R})$ and
\[\langle I'(u),\varphi\rangle=\int_{\mathbb{R}^{3}}\bigl(|\nabla u|^{p-2}\nabla u\nabla \varphi+V(x)|u|^{p-2}u\varphi\bigr)\mathrm{d}x+\int_{\mathbb{R}^{3}}\phi_{u}|u|^{p-2}u\varphi \mathrm{d}x-\int_{\mathbb{R}^{3}}f(u)\varphi \mathrm{d}x.\]
It is easy to verify that $(u,\phi_{u})\in E\times D^{1,2}(\mathbb{R}^{3})$ is a solution of \eqref{1.1} if and only if $u\in E$ is a critical point of $I$.

\section{Proof of Theorem 1.1}

In this section, we prove the existence of sign-changing solutions to system \eqref{1.1} in the case $\mu>2p$, working with \eqref{2.1}.

\subsection{Properties of operator $A$}

We introduce an auxiliary operator $A$, which will be used to construct the descending flow for the functional $I$. Precisely, the operator $A$ is defined as follows: $A:E\to E$, 
\[A(u)=\bigl(-\Delta_{p}+V(x)h_{p}(\cdot)+\phi_{u}h_{p}(\cdot)\bigr)^{-1}f(u),\qquad u\in E,\]
where, $h_{p}(t)=|t|^{p-2}t$. For any $u\in E, v=A(u)\in E$ is the unique solution to the equation
\begin{equation}
	-\Delta_{p}v+V(x)|v|^{p-2}v+\phi_{u}|v|^{p-2}v=f(u),\qquad v\in E.\label{3.1}\tag{3.1}
\end{equation}
Clearly, the three statements are equivalent: $u$ is a solution of \eqref{2.1}, $u$ is a critical point of $I$, and $u$ is a fixed point of $A$. In order to set up our variational framework, we need the following result from \cite{bartsch2004superlinear}. 

\begin{lem}\label{lemma3.1}
	There exist positive constants $d_{i},i=1,2,3,4$, such that for all $x,y\in\mathbb{R}^{N}$,
	\begin{gather}
		||x|^{p-2}x-|y|^{p-2}y|\leqslant d_{1}\bigl(|x|+|y|\bigr)^{p-2}|x-y|,\notag\\
		\bigl(|x|^{p-2}x-|y|^{p-2}y\bigr)(x-y)\geqslant d_{2}\bigl(|x|+|y|\bigr)^{p-2}|x-y|^{2},\notag\\
		||x|^{p-2}x-|y|^{p-2}y|\leqslant d_{3}|x-y|^{p-1}\qquad \text{if}\;1<p\leqslant2,\notag\\
		\bigl(|x|^{p-2}x-|y|^{p-2}y\bigr)(x-y)\geqslant d_{4}|x-y|^{p}\qquad \text{if}\;p\geqslant2.\notag
	\end{gather}
\end{lem}

Moreover, this operator $A$ is merely continuous, as stated in the next lemma.
\begin{lem}\label{lemma3.2}
	The operator $A$ is well defined and is continuous and compact.
\end{lem}
\begin{proof}
	Let $u\in E$ and define
	\[J(w)=\frac{1}{p}\int_{\mathbb{R}^{3}}\bigl(|\nabla w|^{p}+V(x)|w|^{p}\bigr)+\frac{1}{p}\int_{\mathbb{R}^{3}}\phi_{u}|w|^{p}-\int_{\mathbb{R}^{3}}f(u)w,\qquad w\in E.\]
	Then $J(w)\in C^{1}(E,\mathbb{R})$. By $(f_{1})-(f_{2})$ and Remark \ref{remark2.2}, $J(w)$ is coercive, bounded below, weakly lower semicontinuous. We will show that $J(w)$ is strictly convex. For all $\varphi\in E$,
	\[J'(w)\varphi=\int_{\Omega_{3}}|\nabla w|^{p-2}\nabla w\nabla\varphi+\int_{\Omega_{3}}V(x)|w|^{p-2}w\varphi+\int_{\Omega_{3}}\phi_{u}|w|^{p-2}w\varphi-\int_{\Omega_{3}}f(u)\varphi.\]
	We obtain
	\begin{align}
		\langle J'(u_{1})-J'(u_{2}),u_{1}-u_{2}\rangle=&\int_{\Omega_{3}}\bigl(|\nabla u_{1}|^{p-2}\nabla u_{1}-|\nabla u_{2}|^{p-2}\nabla u_{2}\bigr)(\nabla u_{1}-\nabla u_{2})\notag\\
		&+\int_{\Omega_{3}}V(x)\bigl(|u_{1}|^{p-2}u_{1}-|u_{2}|^{p-2}u_{2}\bigr)(u_{1}-u_{2})\notag\\
		&+\int_{\Omega_{3}}\bigl(\phi_{u_{1}}|u_{1}|^{p-2}u_{1}-\phi_{u_{2}}|u_{2}|^{p-2}u_{2}\bigr)(u_{1}-u_{2}).\notag
	\end{align}
	If $u_{1}(x)\equiv u_{2}(x)\equiv0$, the proof is completed. We denote $\Omega_{1}:=\{x\in\mathbb{R}^{3}\bigl||\nabla u_{1}|+\nabla u_{2}|>0\},\Omega_{2}:=\{x\in\mathbb{R}^{3}\bigl||u_{1}|+u_{2}|>0\}$. By Lemma \ref{lemma3.1}, we have
	\begin{align}
		\langle J'(u_{1})-J'(u_{2}),u_{1}-u_{2}\rangle\geqslant&C\biggl(\int_{\Omega_{1}}\frac{|\nabla u_{1}-\nabla u_{2}|^{2}}{\bigl(|\nabla u_{1}|+|\nabla u_{2}|\bigr)^{2-p}}+\int_{\Omega_{2}}V(x)\frac{|u_{1}-u_{2}|^{2}}{\bigl(|u_{1}|+|u_{2}|\bigr)^{2-p}}\notag\\
		&+\int_{\Omega_{3}}\bigl(\phi_{u_{1}}|u_{1}|^{p-2}u_{1}-\phi_{u_{2}}|u_{2}|^{p-2}u_{2}\bigr)(u_{1}-u_{2})\biggr).\notag
	\end{align}
	Let $I_{1}=\int_{\Omega_{3}}\bigl(\phi_{u_{1}}|u_{1}|^{p-2}u_{1}-\phi_{u_{2}}|u_{2}|^{p-2}u_{2}\bigr)(u_{1}-u_{2})$. We get
	\[4\pi I_{1}
	\geqslant D\bigl(|u_{1}|^{p},|u_{1}|^{p}\bigr)+D\bigl(|u_{2}|^{p},|u_{2}|^{p}\bigr)-D\bigl(|u_{1}|^{p},|u_{1}|^{p-1}|u_{2}|\bigr)-D\bigl(|u_{2}|^{p},|u_{2}|^{p-1}|u_{1}|\bigr).\]
	Furthermore, by Remark \ref{remark2.3} and Young's inequality, we obtain
	\begin{align}
		D\bigl(|u_{1}|^{p},|u_{1}|^{p-1}|u_{2}|\bigr)&\leqslant D\biggl(|u_{1}|^{p},\frac{|u_{1}|^{p}}{\frac{p}{p-1}}+\frac{|u_{2}|^{p}}{p}\biggr)\notag\\
		&=\frac{p-1}{p}D\bigl(|u_{1}|^{p},|u_{1}|^{p}\bigr)+\frac{1}{p}D\bigl(|u_{1}|^{p},|u_{2}|^{p}\bigr).\notag
	\end{align}
	Similarly, we get
	\[D\bigl(|u_{2}|^{p},|u_{2}|^{p-1}|u_{1}|\bigr)\leqslant\frac{p-1}{p}D\bigl(|u_{2}|^{p},|u_{2}|^{p}\bigr)+\frac{1}{p}D\bigl(|u_{2}|^{p},|u_{1}|^{p}\bigr).\]
	Thus, by Remark \ref{remark2.3}, we get
	\begin{align}
		4\pi I_{1}=&\frac{1}{p}\biggl(D\bigl(|u_{1}|^{p},|u_{1}|^{p}\bigr)+D\bigl(|u_{2}|^{p},|u_{2}|^{p}\bigr)-2D\bigl(|u_{1}|^{p},|u_{2}|^{p}\bigr)\biggr)\notag\\
		\geqslant&\frac{1}{p}\biggl(D\bigl(|u_{1}|^{p},|u_{1}|^{p}\bigr)+D\bigl(|u_{2}|^{p},|u_{2}|^{p}\bigr)-2\sqrt{D\bigl(|u_{1}|^{p},|u_{1}|^{p}\bigr)D\bigl(|u_{2}|^{p},|u_{2}|^{p}\bigr)}\biggr)\notag\\
		\geqslant&\frac{1}{p}\biggl(\sqrt{D\bigl(|u_{1}|^{p},|u_{1}|^{p}\bigr)}-\sqrt{D\bigl(|u_{2}|^{p},|u_{2}|^{p}\bigr)}\biggr)^{2}\notag\\
		\geqslant&0.\notag
	\end{align}
	Therefore, 
	\[\langle J'(u_{1})-J'(u_{2}),u_{1}-u_{2}\rangle\geqslant0.\]
	We know that $\langle J'(u_{1})-J'(u_{2}),u_{1}-u_{2}\rangle>0$ if $u_{1}\neq u_{2}$. Then $J(w)$ is strictly convex. Thus, $J(w)$ admits a unique minimizer $v=A(u)\in E$, which is the unique solution to \eqref{3.1}. Moreover, $A$ maps bounded sets into bounded sets.
	
	In order to prove that $A$ is continuous, we assume that ${u_{n}}\subset E$ with $u_{n}\to u$ in $E$ as $n\to\infty$. For the sake of brevity, we denote $A(u_{n})=v_{n}$ and $A(u)=v$, then
	\[-\Delta_{p}v_{n}+V(x)|v_{n}|^{p-2}v_{n}+\phi_{u_{n}}|v_{n}|^{p-2}v_{n}=f(u_{n}),\label{3.2}\tag{3.2}\]
	and
	\[-\Delta_{p}v+V(x)|v|^{p-2}v+\phi_{u}|v|^{p-2}v=f(u).\label{3.3}\tag{3.3}\]
	We need to show $\|v_{n}-v\|_{E}\to0$ as $n\to\infty$. If $v\equiv0$, by \eqref{3.3} and $(f_{3})$, we know that $u\equiv0$, furthermore, we have $u_{n}\to0$ and $v_{n}\to v\equiv0$. If $v\not\equiv0$, we denote  $\Omega_{3}:=\bigl\{x\in\mathbb{R}^{3}\big||\nabla v_{n}|+|\nabla v|>0\bigr\},\Omega_{4}:=\bigl\{x\in\mathbb{R}^{3}\big||v_{n}|+|v|>0\bigr\}$. Note that
	\begin{align}
		\|v_{n}-v\|_{E}^{p}=&\int_{\mathbb{R}^{3}}\bigl(|\nabla(v_{n}-v)|^{p}+V(x)|v_{n}-v|^{p}\bigr)\notag\\
		=&\int_{\Omega_{3}}\frac{|\nabla v_{n}-\nabla v|^{p}}{(\bigl|\nabla v_{n}|+|\nabla v|\bigr)^{\frac{p(2-p)}{2}}}\bigl(|\nabla v_{n}|+|\nabla v|\bigr)^{\frac{p(2-p)}{2}}\notag\\
		&+\int_{\Omega_{4}}V(x)\frac{|v_{n}- v|^{p}}{\bigl(|v_{n}|+|v|\bigr)^{\frac{p(2-p)}{2}}} \bigl(|v_{n}|+|v|\bigr)^{\frac{p(2-p)}{2}}.\label{3.4}\tag{3.4}
	\end{align}
	By H\"older's inequality, we have
	\begin{align}
		\|v_{n}-v\|_{E}^{p}\leqslant&\biggl(\int_{\Omega_{3}}\frac{|\nabla v_{n}-\nabla v|^{2}}{\bigl(|\nabla v_{n}|+|\nabla v|\bigr)^{2-p}}\biggr)^{\frac{p}{2}}\biggl(\int_{\Omega_{3}}\bigl(|\nabla v_{n}|+|\nabla v|\bigr)^{p}\biggr)^{\frac{2-p}{2}}\notag\\
		&+\biggl(\int_{\Omega_{4}}V(x)\frac{|v_{n}- v|^{2}}{\bigl(|v_{n}|+|v|\bigr)^{2-p}}\biggr)^{\frac{p}{2}} \biggl(\int_{\Omega_{4}}V(x)\bigl(|v_{n}|+|v|\bigr)^{p}\biggr)^{\frac{2-p}{2}}\notag\\
		\leqslant&\biggl(\int_{\Omega_{3}}\frac{|\nabla v_{n}-\nabla v|^{2}}{\bigl(|\nabla v_{n}|+|\nabla v|\bigr)^{2-p}}+\int_{\Omega_{4}}V(x)\frac{|v_{n}- v|^{2}}{\bigl(|v_{n}|+|v|\bigr)^{2-p}}\biggr)^{\frac{p}{2}}\notag\\
		&\biggl(\int_{\mathbb{R}^{3}}\bigl(|\nabla v_{n}|+|\nabla v|\bigr)^{p}+\int_{\mathbb{R}^{3}}V(x)\bigl(|v_{n}|+|v|\bigr)^{p}\biggr)^{\frac{2-p}{2}}.\notag
	\end{align}
	Together with \eqref{3.2}, \eqref{3.3} and Lemma \ref{lemma3.1}, we get
	\begin{align}
		\|v_{n}-v\|_{E}^{p}
		\leqslant&\biggl(\int_{\Omega_{3}}\langle |\nabla v_{n}|^{p-2}v_{n}-|\nabla v|^{p-2}v,\nabla v_{n}-\nabla v\rangle\notag\\
		&+\int_{\Omega_{4}}\langle V(x)|v_{n}|^{p-2}v_{n}-V(x)|v|^{p-2}v,v_{n}-v\rangle\biggr)^{\frac{p}{2}}\notag\\
		&\cdot C\biggl(\int_{\mathbb{R}^{3}}\bigl(|\nabla v_{n}|^{p}+V(x)|v_{n}|^{p}\bigr)+\int_{\mathbb{R}^{3}}\bigl(|\nabla v|^{p}+V(x)|v|^{p}\bigr)\biggr)^{\frac{2-p}{2}}\notag\\
		\leqslant&C\biggl(\int_{\mathbb{R}^{3}}\langle f(u_{n})-f(u),v_{n}-v\rangle+\int_{\mathbb{R}^{3}}\langle \phi_{u}|v|^{p-2}v-\phi_{u_{n}}|v_{n}|^{p-2}v_{n},v_{n}-v\rangle\biggr)^{\frac{p}{2}}\notag\\
		&\bigl(\|v_{n}\|_{E}^{p}+\|v\|_{E}^{p}\bigr)^{\frac{2-p}{2}}.\label{3.5}\tag{3.5}
	\end{align}
	Indeed, it follows from $(f_{1})$ and $(f_{2})$ that for any $\delta>0$, there exists $C_{\delta}>0$ such that
	\[|f(s)|\leqslant\delta|s|^{p-1}+C_{\delta}|s|^{q-1},\qquad s\in\mathbb{R}^{3}.\label{3.6}\tag{3.6}\]
	Combining Remark \ref{remark2.1}, Remark \ref{remark2.2}, H\"older's inequality and Sobolev's inequality, we obtain
	\begin{align}
		&\int_{\mathbb{R}^{3}}\bigl|\phi_{u}|v|^{p-2}v-\phi_{u_{n}}|v_{n}|^{p-2}v_{n}\bigr|\bigl|v_{n}-v\bigr|\notag\\
		\leqslant&\biggl(\int_{\mathbb{R}^{3}}|\phi_{u}|^{6}\biggr)^{\frac{1}{6}}\biggl(\int_{\mathbb{R}^{3}}\bigl||v|^{p-2}v(v_{n}-v)\bigr|^{\frac{6}{5}}\biggr)^{\frac{5}{6}}\notag\\
		&+\biggl(\int_{\mathbb{R}^{3}}|\phi_{u_{n}}|^{6}\biggr)^{\frac{1}{6}}\biggl(\int_{\mathbb{R}^{3}}\bigl||v_{n}|^{p-2}v_{n}(v_{n}-v)\bigr|^{\frac{6}{5}}\biggr)^{\frac{5}{6}}\notag\\
		\leqslant&C\|\phi_{u}\|_{D^{1,2}}\biggl(\int_{\mathbb{R}^{3}}|u|^{p}\biggr)^{\frac{p-1}{p}}\biggl(\int_{\mathbb{R}^{3}}|v_{n}-v|^{\frac{6p}{6-p}}\biggr)^{\frac{6-p}{6p}}\notag\\
		&+C\|\phi_{u_{n}}\|_{D^{1,2}}\biggl(\int_{\mathbb{R}^{3}}|u_{n}|^{p}\biggr)^{\frac{p-1}{p}}\biggl(\int_{\mathbb{R}^{3}}|v_{n}-v|^{\frac{6p}{6-p}}\biggr)^{\frac{6-p}{6p}}\notag\\
		\leqslant&C\bigl(\|u_{n}\|_{E}^{2p-1}+\|u\|_{E}^{2p-1}\bigr)\|v_{n}-v\|_{\frac{6p}{6-p}}.\label{3.7}\tag{3.7}
	\end{align}
	Using $(f_{2})$ and H\"older's inequality, one sees that
	\begin{align}
		\int_{\mathbb{R}^{3}}|f(u_{n})-f(u)||v_{n}-v|
		&\leqslant \int_{\mathbb{R}^{3}}C\bigl||u_{n}|^{q-1}-|u|^{q-1}\bigr||v_{n}-v|\notag\\
		&\leqslant C\biggl(\int_{\mathbb{R}^{3}}|u_{n}-u|^{q}\biggr)^{\frac{q-1}{q}} \biggl(\int_{\mathbb{R}^{3}}|v_{n}-v|^{q}\biggr)^{\frac{1}{q}}\notag\\
		&=C\|u_{n}-u\|_{q}^{q-1}\|v_{n}-v\|_{q}.\label{3.8}\tag{3.8}
	\end{align}
	By \eqref{3.5}, \eqref{3.7} and \eqref{3.8}, we get
	\begin{align}
		\|v_{n}-v\|_{E}^{p}\leqslant
		&C\biggl(\|u_{n}-u\|_{q}^{q-1}\|v_{n}-v\|_{q}+\bigl(\|u_{n}\|_{E}^{2p-1}+\|u\|_{E}^{2p-1}\bigr)\|v_{n}-v\|_{\frac{6p}{6-p}}\biggr)^{\frac{p}{2}}\notag\\
		&\bigl(\|v_{n}\|_{E}+\|v\|_{E}\bigr)^{\frac{(2-p)p}{2}}\notag\\
		\leqslant&C_{1}\|u_{n}-u\|_{E}^{\frac{(q-1)p}{2}}\|v_{n}-v\|_{E}^{\frac{p}{2}} \bigl(\|v_{n}\|_{E}+\|v\|_{E}\bigr)^{\frac{(2-p)p}{2}}.\notag
	\end{align}
	Thus,
	\[\|v_{n}-v\|_{E}\leqslant C_{1}\|u_{n}-u\|_{E}^{q-1}\bigl(\|v_{n}\|_{E}+\|v\|_{E}\bigr)^{2-p}.\]
	Therefore, for $p>1$, $\|v_{n}-v\|_{E}\to 0$ as $\|u_{n}-u\|_{E}\to 0$.
	
	Finally, we show that $A$ is compact. Let $\{u_{n}\}\subset E$  be a bounded sequence. Then $\{v_{n}\}=\{A(u_{n})\}\subset E$ is a bounded sequence. Passing to a subsequence, we may assume that $u_{n}\rightharpoonup u$ and $v_{n}\rightharpoonup v$ in $E$ and strongly convergent in $L^{q}(\mathbb{R}^{3})$ as $n\to\infty$, Consider the identity
	\[\int_{\mathbb{R}^{3}}\biggl(|\nabla v_{n}|^{p-2}\nabla v_{n}\nabla \xi+V(x)|v_{n}|^{p-2}v_{n}\xi+\phi_{u_{n}} |v_{n}|^{p-2}v_{n}\xi\biggr)=\int_{\mathbb{R}^{3}}f(u_{n})\xi,\;\;\xi\in E.\label{3.9}\tag{3.9}\]
	Taking limit as $n\to\infty$ in \eqref{3.9} yields
	\[\int_{\mathbb{R}^{3}}\biggl(|\nabla v|^{p-2}\nabla v\nabla\xi+V(x)|v|^{p-2}v\xi+\phi_{u}|v|^{p-2}v\xi\biggr)=\int_{\mathbb{R}^{3}}f(u)\xi,\qquad\xi\in E.\]
	This means $v=A(u)$ and in the same way as above, for $p>1$, $\|v_{n}-v\|_{E}\to 0$, i.e., $A(u_{n})\to A(u)$ in $E$ as $\|u_{n}-u\|_{E}\to 0$.
\end{proof}

Obviously, if $f$ is odd then $A$ is odd. Then we denote the set of fixed points of $A$ by $K$, which is exactly the set of critical points of $I$. 

\begin{lem}\label{lemma3.3}
	$(1)$ There exist $1<p\leqslant2,a_{1}>0$ and $a_{2}>0$ such that
	\[\langle I'(u),u-A(u)\rangle\geqslant a_{1}\|u-A(u)\|_{E}^{2}(\|u\|_{E}+\|A(u)\|_{E})^{p-2}\]
	and
	\[\|I'(u)\|_{E^{*}}\leqslant a_{2}\|u-A(u)\|_{E}^{p-1}\bigl(1+\|u\|_{E}^{p}\bigr)\]
	hold for every $u\in E_{0}:=E\backslash K$.\\
	$(2)$There exist $p\geqslant2,a_{1}>0$ and $a_{2}>0$ such that
	\[\langle I'(u),u-A(u)\rangle\geqslant a_{1}\|u-A(u)\|_{E}^{p}.\]
	and
	\[\|I'(u)\|_{E^{*}}\leqslant a_{2}\|u-A(u)\|_{E}\bigl(\|u\|_{E}+\|A(u)\|_{E}\bigr)^{p-2}\bigl(1+\|u\|_{E}^{p}\bigr)\]
	hold for every $u\in E$.
\end{lem}
\begin{proof}
	Since $A(u)$ is the solution of equation (\ref{3.1}), we see that
	\begin{align}
		&\;\langle I'(u),u-A(u)\rangle\notag\\
		=&\int_{\mathbb{R}^{3}}\biggl(|\nabla u|^{p-2}\nabla u\nabla\bigl(u-A(u)\bigr)+V(x)|u|^{p-2}u\bigl(u-A(u)\bigr)+\phi_{u}|u|^{p-2}u\bigl(u-A(u)\bigr)\biggr)\notag\\
		&-\int_{\mathbb{R}^{3}}f(u)\bigl(u-A(u)\bigr)\notag\\
		=&\int_{\mathbb{R}^{3}}\bigl(|\nabla u|^{p-2}\nabla u-|\nabla v|^{p-2}\nabla v\bigr)(\nabla u-\nabla v)+\int_{\mathbb{R}^{3}}V(x)\bigl(|u|^{p-2}u-|v|^{p-2}v\bigr)(u-v)\notag\\
		&+\int_{\mathbb{R}^{3}}\phi_{u}\bigl(|u|^{p-2}u-|v|^{p-2}v\bigr)(u-v).\label{3.10}\tag{3.10}
	\end{align}
	Together with \eqref{3.10} and Lemma \ref{lemma3.1}, we have
	\begin{align}
		\langle I'(u),u-A(u)\rangle
		\geqslant&C\biggl(\int_{\mathbb{R}^{3}}\bigl(|\nabla u|+|\nabla v|\bigr)^{p-2}|\nabla u-\nabla v|^{2}+\int_{\mathbb{R}^{3}}V(x)\bigl(|u|+|v|\bigr)^{p-2}|u-v|^{2}\notag\\
		&+\int_{\mathbb{R}^{3}}\phi_{u}\bigl(|u|+|v|\bigr)^{p-2}|u-v|^{2}\biggr).\label{3.11}\tag{3.11}
	\end{align}
	By using \eqref{3.4} and H\"older's inequality, we get
	\begin{align}
		\|u-v\|_{E}^{p}
		\leqslant&\biggl(\int_{\mathbb{R}^{3}}|\nabla u-\nabla v|^{2}\bigl(|\nabla u|+|\nabla v|\bigr)^{p-2}\biggr)^{\frac{p}{2}}\biggl(\int_{\mathbb{R}^{3}}\bigl(|\nabla u|+|\nabla v|\bigr)^{p}\biggr)^{\frac{2-p}{2}}\notag\\
		&+\biggl(\int_{\mathbb{R}^{3}}V(x)|u-v|^{2}\bigl(|u|+|v|\bigr)^{p-2}\biggr)^{\frac{p}{2}}\biggl(\int_{\mathbb{R}^{3}}V(x)\bigl(|u|+|v|\bigr)^{p}\biggr)^{\frac{2-p}{2}}.\notag
	\end{align}
	Further we have
	\begin{align}
		\|u-v\|_{E}^{p}
		\leqslant&C_{3}\biggl(\int_{\mathbb{R}^{3}}|\nabla u-\nabla v|^{2}\bigl(|\nabla u|+|\nabla v|\bigr)^{p-2}+\int_{\mathbb{R}^{3}}V(x)|u-v|^{2}\bigl(|u|+|v|\bigr)^{p-2}\biggr)^{\frac{p}{2}}\notag\\
		&\bigl(\|u\|_{E}+\|v\|_{E}\bigr)^{\frac{p(2-p)}{2}}.\notag
	\end{align}
	Set $\int_{\mathbb{R}^{3}}|\nabla u-\nabla v|^{2}\bigl(|\nabla u|+|\nabla v|\bigr)^{p-2}+\int_{\mathbb{R}^{3}}V(x)|u-v|^{2}\bigl(|u|+|v|\bigr)^{p-2}=e$, then
	\[\|u-v\|_{E}^{p}\leqslant C_{3}e^{\frac{p}{2}}\bigl(\|u\|_{E}+\|v\|_{E}\bigr)^{\frac{p(2-p)}{2}}.\]
	Furthermore,
	\[e\geqslant C_{4}\|u-v\|_{E}^{2}\bigl(\|u\|_{E}+\|v\|_{E}\bigr)^{p-2}.\]
	Together with \eqref{3.11}, wo obtain
	\[\langle I'(u),u-A(u)\rangle\geqslant C\biggl(e+\int_{\mathbb{R}^{3}}\phi_{u}\bigl(|u|+|v|\bigr)^{p-2}|u-v|^{2}\biggr).\]
	Therefore, if $1<p\leqslant2$, for all $u\in E$
	\begin{align}
		\langle I'(u),u-A(u)\rangle&\geqslant C\biggl(C_{4}\|u-v\|_{E}^{2}\bigl(\|u\|_{E}+\|v\|_{E}\bigr)^{p-2}+\int_{\mathbb{R}^{3}}\phi_{u}\bigl(|u|+|v|\bigr)^{p-2}|u-v|^{2}\biggr)\notag\\
		&\geqslant C_{5}\|u-A(u)\|_{E}^{2}\bigl(\|u\|_{E}+\|A(u)\|_{E}\bigr)^{p-2}.\notag
	\end{align}
	Using \eqref{3.5} and Lemma \ref{lemma3.1}, and the fact that $p\geqslant2$, one sees that
	\begin{align}
		\langle I'(u),u-A(u)\rangle
		&\geqslant a_{1}\biggl(\int_{\mathbb{R}^{3}}|\nabla u-\nabla v|^{p}+\int_{\mathbb{R}^{3}}V(x)|u-v|^{p}+\int_{\mathbb{R}^{3}}\phi_{u}|u-v|^{p}\biggr)\notag\\
		&\geqslant a_{1}\|u-v\|_{E}^{p}.\notag 
	\end{align}	
	Similar to \eqref{3.10}, we get
	\begin{align}
		\langle I'(u),\varphi\rangle=&\int_{\mathbb{R}^{3}}\bigl(|\nabla u|^{p-2}\nabla u-|\nabla v|^{p-2}\nabla v\bigr)\nabla\varphi+\int_{\mathbb{R}^{3}}V(x)\bigl(|u|^{p-2}u-|v|^{p-2}v\bigr)\varphi\notag\\
		&+\int_{\mathbb{R}^{3}}\phi_{u}\bigl(|u|^{p-2}u-|v|^{p-2}v\bigr)\varphi.\notag
	\end{align}
	By Lemma \ref{lemma3.1} and the fact that $1<p\leqslant2$, we obtain
	\begin{align}
		\langle I'(u),\varphi\rangle\leqslant&\int_{\mathbb{R}^{3}}\bigl||\nabla u|^{p-2}\nabla u-|\nabla v|^{p-2}\nabla v\bigl||\nabla\varphi|+\int_{\mathbb{R}^{3}}V(x)\bigl||u|^{p-2}u-|v|^{p-2}v\bigl||\varphi|\notag\\
		&+\int_{\mathbb{R}^{3}}\phi_{u}\bigl(|u|^{p-2}u-|v|^{p-2}v\bigr)\varphi\notag\\
		\leqslant&C\biggl(\int_{\mathbb{R}^{3}}\bigl(d_{3}|\nabla u-\nabla v|^{p-1}\bigr)^{\frac{p}{p-1}}+\int_{\mathbb{R}^{3}}V(x)\bigl(d_{4}|u-v|^{p-1}\bigr)^{\frac{p}{p-1}}\biggr)^{\frac{p-1}{p}}\|\varphi\|_{E}\notag\\
		&+\int_{\mathbb{R}^{3}}\phi_{u}\bigl(|u|^{p-2}u-|v|^{p-2}v\bigr)\varphi.\notag
	\end{align}
Thus,
\begin{align}
	\langle I'(u),\varphi\rangle\leqslant& C_{1}\biggl(\int_{\mathbb{R}^{3}}|\nabla(u-v)|^{p}+V(x)|u-v|^{p}\biggr)^{\frac{p-1}{p}}\|\varphi\|_{E}+\int_{\mathbb{R}^{3}}\phi_{u}\bigl(|u|^{p-2}u-|v|^{p-2}v\bigr)\varphi\notag\\
	\leqslant& C_{1}\|u-v\|_{E}^{p-1}\|\varphi\|_{E}+\int_{\mathbb{R}^{3}}\phi_{u}\bigl(|u|^{p-2}u-|v|^{p-2}v\bigr)\varphi.\label{3.13}\tag{3.13}
\end{align}
	Therefore, if $1<p\leqslant2$, using Remark \ref{remark2.1}, Remark \ref{remark2.2} and Lemma \ref{lemma3.1},\;H\"older's inequality and Sobolev's inequality, it follows that
	\begin{align}
		\int_{\mathbb{R}^{3}}\phi_{u}\bigl(|u|^{p-2}u-|v|^{p-2}v\bigr)\varphi\leqslant&C\int_{\mathbb{R}^{3}}\phi_{u}|u-v|^{p-1}\varphi\notag\\
		\leqslant& C\biggl(\int_{\mathbb{R}^{3}}|\phi_{u}|^{6}\biggr)^{\frac{1}{6}}\biggl(\int_{\mathbb{R}^{3}}|u-v|^{p}\biggr)^{\frac{p-1}{p}}\biggl(\int_{\mathbb{R}^{3}}|\varphi|^{\frac{6}{6-p}}\biggr)^{\frac{6-p}{6}}\notag\\
		\leqslant&C_{1}\|u\|_{E}^{p}\|u-v\|_{E}^{p-1}\|\varphi\|_{E}.\label{3.14}\tag{3.14}
	\end{align}
	By \eqref{3.13} and \eqref{3.14}, we obtain
	\begin{align}
		\langle I'(u),\varphi\rangle&\leqslant C_{1}\|u-v\|_{E}^{p-1}\|\varphi\|_{E}+C\|u\|_{E}^{p}\|u-v\|_{E}^{p-1}\|\varphi\|_{E}\notag\\
		&\leqslant C\|u-v\|_{E}^{p-1}\bigl(1+\|u\|_{E}^{p}\bigr)\|\varphi\|_{E}.\notag
	\end{align}
	Thus,
	\[\|I'(u)\|_{E^{*}}=\sup_{\varphi\in E}\frac{\|\langle I'(u),\varphi\rangle\|}{\|\varphi\|_{E}}\leqslant C\|u-v\|_{E}^{p-1}\bigl(1+\|u\|_{E}^{p}\bigr).\]
	If $p\geqslant2$, it follows from Lemma \ref{lemma3.1},\;H\"older's inequality and Sobolev's inequality that
	\begin{align}
		\int_{\mathbb{R}^{3}}\phi_{u}\bigl(|u|^{p-2}-|v|^{p-2}v\bigr)\varphi
		\leqslant&C\int_{\mathbb{R}^{3}}\phi_{u}\bigl(|u|+|v|\bigr)^{p-2}|u-v|\varphi\notag\\
		\leqslant& C\biggl(\int_{\mathbb{R}^{3}}|\phi_{u}|^{6}\biggr)^{\frac{1}{6}}\biggl(\int_{\mathbb{R}^{3}}\bigl(|u|+|v|\bigr)^{\frac{6p(p-2)}{5p-6}}\varphi^{\frac{6p}{5p-6}}\biggr)^{\frac{5p-6}{6p}}\notag\\
		&\biggl(\int_{\mathbb{R}^{3}}|u-v|^{p}\biggr)^{\frac{1}{p}}\notag\\
		\leqslant& C\|\phi_{u}\|_{6}\biggl(\int_{\mathbb{R}^{3}}\bigl(|u|+|v|\bigr)^{p}\biggr)^{\frac{p-2}{p}}\biggl(\int_{\mathbb{R}^{3}}|\varphi|^{\frac{6p}{6-p}}\biggr)^{\frac{6-p}{6p}}\|u-v\|_{p}.\notag
	\end{align}
	Then Remark \ref{remark2.1} implies that
	\begin{align}
		\int_{\mathbb{R}^{3}}\phi_{u}\bigl(|u|^{p-2}-|v|^{p-2}v\bigr)\varphi\leqslant&C\|u\|_{E}^{p}\|u-v\|_{p}\||u|+|v|\|_{p}^{p-2}\|\varphi\|_{\frac{6p}{6-p}}\notag\\
		\leqslant&C\|u\|_{E}^{p}\|u-v\|_{E}\bigl(\|u\|_{E}+\|v\|_{E}\bigr)^{p-2}\|\varphi\|_{E}.\label{3.15}\tag{3.15}
	\end{align}
	Similar to \eqref{3.10}, by Lemma \ref{lemma3.1}, we obtain
	\begin{align}
		\langle& I'(u),\varphi\rangle\notag\\
		\leqslant&\biggl(\int_{\mathbb{R}^{3}}\bigl(d_{1}\bigl(|\nabla u|+|\nabla v|\bigr)^{p-2}|\nabla u-\nabla v|\bigr)^{\frac{p}{p-1}}+\int_{\mathbb{R}^{3}}\bigl(V(x)d_{2}(|u|+|v|)^{p-2}|u-v|\bigr)^{\frac{p}{p-1}}\biggr)^{\frac{p-1}{p}}\notag\\
		&\|\varphi\|_{E}+\int_{\mathbb{R}^{3}}\phi_{u}\bigl(|u|^{p-2}u-|v|^{p-2}v\bigr)\varphi\notag\\
		\leqslant&C_{1}\biggl[\biggl(\int_{\mathbb{R}^{3}}\bigl(|\nabla u|+|\nabla v|\bigr)^{p-1}\biggr)^{\frac{p-2}{p-1}}\biggl(\int_{\mathbb{R}^{3}}|\nabla(u-v)|^{p}\biggr)^{\frac{1}{p-1}}+\biggl(\int_{\mathbb{R}^{3}}V(x)\bigl(|u|+|v|\bigr)^{p-1}\biggr)^{\frac{p-2}{p-1}}\notag\\
		&\biggl(\int_{\mathbb{R}^{3}}V(x)|u-v|^{p}\biggr)^{\frac{1}{p-1}}\biggr]^{\frac{p-1}{p}}\|\varphi\|_{E}+\int_{\mathbb{R}^{3}}\phi_{u}\bigl(|u|^{p-2}u-|v|^{p-2}v\bigr)\varphi.\notag
	\end{align}
	Thus,
	\begin{align}
		\langle I'(u),\varphi\rangle
		\leqslant& C_{3}\biggl\{\biggl(\int_{\mathbb{R}^{3}}|\nabla(u-v)|^{p}+V(x)|u-v|^{p}\biggr)^{\frac{1}{p}}\biggl[\biggl(\int_{\mathbb{R}^{3}}|\nabla u|^{p}+V(x)|u|^{p}\biggr)^{\frac{1}{p}}\notag\\
		&+\biggl(\int_{\mathbb{R}^{3}}|\nabla v|^{p}+V(x)|v|^{p}\biggr)^{\frac{1}{p}}\biggr]^{p-2}\biggr\}\|\varphi\|_{E}+\int_{\mathbb{R}^{3}}\phi_{u}\bigl(|u|^{p-2}u-|v|^{p-2}v\bigr)\varphi\notag\\
		\leqslant&C\|u-v\|_{E}\bigl(\|u\|_{E}+\|v\|_{E}\bigr)^{p-2}\bigl(1+\|u\|_{E}^{p}\bigr)\|\varphi\|_{E}.\notag
	\end{align}
	Therefore, if $p\geqslant2$,
	\[\|I'(u)\|_{E^{*}}=\sup_{\varphi\in E}\frac{\|\langle I'(u),\varphi\rangle\|}{\|\varphi\|_{E}}\leqslant C\|u-v\|_{E}\bigl(\|u\|_{E}+\|v\|_{E}\bigr)^{p-2}\bigl(1+\|u\|_{E}^{p}\bigr).\]
\end{proof}
From Lemma \ref{lemma3.3}, it is clear that the set of fixed points of $A$ is the same as the set of critical points of $I$.
\begin{lem}\label{lemma3.4}
	If $\{u_{n}\}\subset E$ is a bounded sequence with $I'(u_{n})\to0$, then $\{u_{n}\}\subset E$  has a convergent subsequence.
\end{lem}
\begin{proof}
	By the boundedness of $\{u_{n}\}\subset E$, we let $u_{n}\rightharpoonup u$ in $E$. Without loss of generality, we assume $u\neq A(u)$. Note that
	\begin{align}
		\langle I'(u_{n})&-I'(u), u_{n}-u\rangle\notag\\
		=&\int_{\mathbb{R}^{3}}\bigl(|\nabla u_{n}|^{p-2}\nabla u_{n}-|\nabla u|^{p-2}\nabla u\bigr)(\nabla u_{n}-\nabla u)\notag\\
		&+\int_{\mathbb{R}^{3}}V(x)\bigl(|u_{n}|^{p-2}u_{n}-|u|^{p-2}u\bigr)(u_{n}-u)\notag\\
		&+\int_{\mathbb{R}^{3}}\bigl(\phi_{u_{n}}|u_{n}|^{p-2}u_{n}-\phi_{u}|u|^{p-2}u\bigr)(u_{n}-u)-\int_{\mathbb{R}^{3}}(f(u_{n})-f(u))(u_{n}-u).\notag
	\end{align}
	Similar to \eqref{3.5}, we get
	\begin{align}
	\|u_{n}-u\|_{E}^{p}\leqslant&C\biggl(\int_{\mathbb{R}^{3}}\bigl(|\nabla u_{n}|^{p-2}\nabla u_{n}-|\nabla u|^{p-2}\nabla u\bigr)(\nabla u_{n}-\nabla u)\notag\\
	&+\int_{\mathbb{R}^{3}}V(x)\bigl(|u_{n}|^{p-2}u_{n}-|u|^{p-2}u\bigr)(u_{n}-u)\biggr)^{\frac{p}{2}}\bigl(\|u_{n}\|_{E}^{p}+\|u\|_{E}^{p}\bigr)^{\frac{2-p}{2}}.\notag
	\end{align}
	Then,
	\begin{align}
		\|u_{n}-u\|_{E}^{p}\leqslant&C\biggl(\langle I'(u_{n})-I'(u), u_{n}-u\rangle+\int_{\mathbb{R}^{3}}\bigl(\phi_{u_{n}}|u_{n}|^{p-2}u_{n}-\phi_{u}|u|^{p-2}u\bigr)(u-u_{n})\notag\\
		&+\int_{\mathbb{R}^{3}}(f(u_{n})-f(u))(u_{n}-u)\biggr)^{\frac{p}{2}}\bigl(\|u_{n}\|_{E}^{p}+\|u\|_{E}^{p}\bigr)^{\frac{2-p}{2}}.\tag{3.16}\label{3.16}
	\end{align}
	Similar to \eqref{3.7}, we have
	\[\int_{\mathbb{R}^{3}}\bigl|\phi_{u_{n}}|u_{n}|^{p-2}u_{n}-\phi_{u}|u|^{p-2}u\bigr|\bigl|u-u_{n}\bigr|\leqslant C_{1}\bigl(\|u_{n}\|_{E}^{2p-1}+\|u\|_{E}^{2p-1}\bigr)\|u_{n}-u\|_{\frac{6p}{6-p}}.\]
	By the boundedness of $\{u_{n}\}\subset E$ and Lemma \ref{lemma3.1}, one sees that
	\[\int_{\mathbb{R}^{3}}\bigl|\phi_{u_{n}}|u_{n}|^{p-2}u_{n}-\phi_{u}|u|^{p-2}u\bigr|\bigl|u_{n}-u\bigr|\leqslant C_{2}\|u_{n}-u\|_{\frac{6p}{6-p}}.\]
	Since $\|u_{n}-u\|_{\frac{6p}{6-p}}\to0$, as $n\to\infty$, we have proved that
	\[\int_{\mathbb{R}^{3}}\bigl|\phi_{u_{n}}|u_{n}|^{p-2}u_{n}-\phi_{u}|u|^{p-2}u\bigr|\bigl|u_{n}-u\bigr|\to0.\]
	Similar to \eqref{3.8}, we get
	\[\int_{\mathbb{R}^{3}}|f(u_{n})-f(u)||u-u_{n}|\leqslant C_{3}\|u_{n}-u\|_{q}^{q}.\]
	Together with Remark \ref{remark2.2}, we obtain
	\[\int_{\mathbb{R}^{3}}|f(u_{n})-f(u)||u-u_{n}|\to0.\]
	Consequently, by $I'(u_{n})\to0$ and \eqref{3.16}, we know that $\|u_{n}-u\|_{E}\to0$.
\end{proof}

\begin{lem}\label{lemma3.5}
	For $a<b,\alpha>0$, there exists $\beta>0$, such that $\|u-A(u)\|_{E}\geqslant\beta$ if $u\in E,I(u)\in[a,b]$ and $\|I'(u)\|_{E^{*}}\geqslant\alpha$.
\end{lem}
\begin{proof}
	By Lemma \ref{lemma3.4}, we prove that $I$ satisfies the (PS) condition. For $u\in E$, by $(f_{3})$, we have
	\begin{align}
		I(u)&-\frac{1}{\mu}\langle u-v,u\rangle_{E}\notag\\
		=&\biggl(\frac{1}{p}-\frac{1}{\mu}\biggr)\|u\|_{E}^{p}+\frac{1}{2p}\int_{\mathbb{R}^{3}}\phi_{u}|u|^{p}-\frac{1}{\mu}\int_{\mathbb{R}^{3}}\phi_{u}|v|^{p-2}vu+\int_{\mathbb{R}^{3}}\biggl(\frac{1}{\mu}f(u)u-F(u)\biggr)\notag\\
		=&\biggl(\frac{1}{p}-\frac{1}{\mu}\biggr)\|u\|_{E}^{p}+\biggl(\frac{1}{2p}-\frac{1}{\mu}\biggr)\int_{\mathbb{R}^{3}}\phi_{u}|u|^{p}+\frac{1}{\mu}\int_{\mathbb{R}^{3}}\phi_{u}u\bigl(|u|^{p-2}u-|v|^{p-2}v\bigr).\notag
	\end{align}
	Then,
	\[\|u\|_{E}^{p}+\int_{\mathbb{R}^{3}}\phi_{u}|u|^{p}\leqslant C\biggl(|I(u)|+\|u\|_{E}\|u-v\|_{E}+\bigg|\int_{\mathbb{R}^{3}}\phi_{u}u(|u|^{p-2}u-|v|^{p-2}v)\biggl|\biggr).\]
	By H\"older's inequality, for any $\varepsilon>0$, there exists $C_{\varepsilon}>0$ such that 
	\begin{align}
		\biggr|\int_{\mathbb{R}^{3}}\phi_{u}u(|u|^{p-2}u-|v|^{p-2}v)\biggl|&\leqslant C\biggl(\int_{\mathbb{R}^{3}}|\phi_{u}|^{p}\biggr)^{\frac{1}{p}}\|u\|_{E}^{p-1}\|u-v\|_{E}^{p-1}\notag\\
		&\leqslant\varepsilon\int_{\mathbb{R}^{3}}|\phi_{u}|^{p}+C_{\varepsilon}\|u\|_{E}^{p-1}\|u-v\|_{E}^{p-1}.\notag
	\end{align}
	Therefore,
	\begin{align}
		\|u\|_{E}^{p}+\int_{\mathbb{R}^{3}}\phi_{u}|u|^{p}\leqslant C\biggl(|I(u)|+\|u\|_{E}\|u-v\|_{E}+\varepsilon\int_{\mathbb{R}^{3}}|\phi_{u}|^{p}+C_{\varepsilon}\|u\|_{E}^{p-1}\|u-v\|_{E}^{p-1}\biggr).\notag
	\end{align}
	Thus,
	\[\|u\|_{E}^{p}\leqslant C\bigl(|I(u)|+\|u\|_{E}\|u-v\|_{E}+\|u\|_{E}^{p}\|u-v\|_{E}^{p}\bigr).\tag{3.17}\label{3.17}\]
	Assume on the contrary that there exists $\{u_{n}\}\subset E$ with $I(u_{n})\in[a,b]$ and $\|I'(u_{n})\|_{E^{*}}\geqslant\alpha$, such that $\|u_{n}-A(u_{n})\|_{E}\to 0$ as $n\to\infty$. Then it follows from \eqref{3.17} that $\{\|u_{n}\|_{E}\}$ is bounded. And by Lemma \ref{lemma3.3}, whether $1<p\leqslant2$ or $p\geqslant 2$, there is $\|I'(u_{n})\|_{E^{*}}\to 0$ as $n\to\infty$, which is a contradiction. The proof is completed.
\end{proof}

\subsection{Invariant subsets of descending flow}

In order to obtain sign-changing solutions, we refer to \cite{bartsch2005nodal,liu2016infinitely}. Precisely, we define the positive and negative cones by
\[P^{+}:=\{u\in E: u\geqslant0\}\;\;\text{and}\;\;P^{-}:=\{u\in E: u\leqslant0\}.\]
For $\varepsilon>0$, set
\[P_{\varepsilon}^{+}:=\{u\in E: \text{dist}(u,P^{+})<\varepsilon\}\;\;\text{and}\;\;P_{\varepsilon}^{-}:=\{u\in E: \text{dist}(u,P^{-})<\varepsilon\},\]
where $\text{dist}(u,P^{\pm})=\inf_{v\in P^{\pm}}\|u-v\|_{E}$. Obviously, $P_{\varepsilon}^{-}=-P_{\varepsilon}^{+}$. Let $W=P_{\varepsilon}^{+}\cup P_{\varepsilon}^{-}$. It is
easy to check that $W$ is an open and symmetric subset of $E$ and $E\backslash W$ contains only sign-changing functions. In the following, we will show that for $\varepsilon$ small, all sign-changing solutions to \eqref{2.1} are contained in $E\backslash W$.

\begin{lem}\label{lemma3.6}
	There exists $\varepsilon_{0}>0$ such that for $\varepsilon\in(0,\varepsilon_{0})$,\\
	$(1)$\;$A(\partial P_{\varepsilon}^{-})\subset P_{\varepsilon}^{-}$ and every nontrivial solution $u\in P_{\varepsilon}^{-}$is negative,\\
	$(2)$\;$A(\partial P_{\varepsilon}^{+})\subset P_{\varepsilon}^{+}$ and every nontrivial solution $u\in P_{\varepsilon}^{+}$ is positive.\\
\end{lem}
\begin{proof}
	We only prove the first one, and the other case is similar. Let $u\in E,v=A(u)$, since $\|u^{-}\|_{q}=\inf_{w\in P^{+}}\|u-w\|_{q},\text{dist}(v,P^{-})\leqslant\|v^{+}\|_{E}$(see \cite{liu2016infinitely}).
	
	We claim that  for any $q\in [\,p,p^{*})$, there exists $m_{q}>0$, $\|u^{\pm}\|_{q}\leqslant m_{q}\text{dist}(u,P^{\mp})$, whose proof is similar to \cite{liu2016infinitely}. By equation \eqref{2.1}, we get
	\begin{align}
		&\int_{\mathbb{R}^{3}}\bigl(|\nabla v|^{p-2}\nabla v\nabla v^{+}+V(x)|v|^{p-2}vv^{+}\bigr)=\int_{\mathbb{R}^{3}}f(u)v^{+}-\int_{\mathbb{R}^{3}}\phi_{u}|v|^{p-2}vv^{+}.\notag
	\end{align}
	Furthermore,
	\begin{equation}
		\int_{\mathbb{R}^{3}}\bigl(|\nabla v^{+}|^{p}+V(x)|v^{+}|^{p}\bigr)
		=\int_{\mathbb{R}^{3}}f(u)v^{+}-\int_{\mathbb{R}^{3}}\phi_{u}|v^{+}|^{p}.\label{3.18}\tag{3.18}
	\end{equation}
	Combining \eqref{3.18} and $(f_{1})$, we have
	\begin{align}
		\text{dist}(v,P^{-})^{p-1}\|v^{+}\|_{E}&\leqslant \int_{\mathbb{R}^{3}}\bigl(|\nabla v^{+}|^{p}+V(x)|v^{+}|^{p}\bigr)\notag\\
		&=\int_{\mathbb{R}^{3}}f(u)v^{+}-\int_{\mathbb{R}^{3}}\phi_{u}|v^{+}|^{p}\notag\\
		&\leqslant \int_{\mathbb{R}^{3}}f(u^{+})v^{+}.\notag
	\end{align}
	Together with \eqref{3.6} and H\"older's inequality, we obtain
	\begin{align}
		\text{dist}(v,P^{-})^{p-1}\|v^{+}\|_{E}\leqslant&\int_{\mathbb{R}^{3}}\bigl(\delta|u^{+}|^{p-1}+C_{\delta}|u^{+}|^{q-1}\bigr)v^{+}\notag\\
		=&\;\delta\|u^{+}\|_{p}^{p-1}\|v^{+}\|_{p}+C_{\delta}\|u^{+}\|_{q}^{q-1}\|v^{+}\|_{q}\notag\\
		\leqslant&\;C\bigl(\delta+C_{\delta}\text{dist}(u,P^{-})^{q-p}\bigr)\text{dist}(u,P^{-})^{p-1}\|v^{+}\|_{E},\notag
	\end{align}
	which further implies that
	\[\text{dist}(v,P^{-})^{p-1}\leqslant C\bigl(\delta+C_{\delta}\text{dist}(u,P^{-})^{q-p}\bigr)\text{dist}(u,P^{-})^{p-1}.\]
	Thus, choosing $\delta$ small enough, there exists $\varepsilon_{0}>0$ such that for $\varepsilon\in(0,\varepsilon_{0})$,
	\[\text{dist}(v,P^{-})^{p-1}\leqslant\biggl(\frac{1}{2}+\frac{1}{4}\biggr)\text{dist}(u,P^{-})^{p-1},\qquad\forall u\in P_{\varepsilon}^{-}.\]
	That means $A(\partial P_{\varepsilon}^{-})\subset P_{\varepsilon}^{-}$. If there exists $u\in P_{\varepsilon}^{-}$ such that $A(u)=u$, then $u\in P^{-}$. If $u\not\equiv0$, by the maximum principle, $u<0$ in $\mathbb{R}^{3}$. 
\end{proof}

Since $A$ is merely continuous, $A$ itself is not applicable to construct a descending flow for $I$, and we have to construct a locally Lipschitz continuous operator $B$ on $E_{0}:=E\backslash K$ which inherits the main properties of $A$.

\begin{lem}\label{lemma3.7}
	There exists a locally Lipschitz continuous operator $B:E_{0}\to E$ such that\\
	$(1)$\;$B(\partial P_{\varepsilon}^{+})\subset P_{\varepsilon}^{+}$ and $B(\partial P_{\varepsilon}^{-})\subset P_{\varepsilon}^{-}$ for $\varepsilon\in (0,\varepsilon_{0});$\\
	$(2)$\;for all $u\in E_{0},$
	\[\frac{1}{2}\|u-B(u)\|_{E}\leqslant\|u-A(u)\|_{E}\leqslant2\|u-B(u)\|_{E};\]
	$(3)$\;for all $u\in E_{0},$ if $1<p\leqslant2,$\\
	\[\langle I'(u),u-B(u)\rangle\geqslant\frac{1}{2}a_{1}\|u-A(u)\|_{E}^{2}(\|u\|_{E}+\|A(u)\|_{E})^{p-2};\]
	if $p\geqslant2,$
	\[\langle I'(u),u-B(u)\rangle\geqslant\frac{1}{2}a_{1}\|u-A(u)\|_{E}^{p};\]
	$(4)$\;if $I$ is even, $A$ is odd, then $B$ is odd.
\end{lem}
\begin{proof}
	The proof is similar to the proofs of \cite[Lemma 4.1]{bartsch2004superlinear} and \cite[Lemma 2.1]{bartsch2005nodal}. We omit the details. 
\end{proof}

\subsection{Existence of one sign-changing solution}
In this subsection, we introduce the critical point theorem \cite[theorem 2.4]{liu2015multiple}. In order to employ Theorem \ref{theorem3.1} to prove the existence of sign-changing solutions to problem \eqref{2.1}, we let $X$ be a Banach space, $J\in C^{1}(X,\mathbb{R}), P,Q\subset X$ be open sets, $M=P\cap Q$,
$\Sigma=\partial P\cap\partial Q$ and $W=P\cup Q$. For $c\in \mathbb{R}$, $K_{c}=\{x\in X: J(x)=c, J'(x)=0\}$ and
$J^{c}=\{x\in X: J(x)\leqslant c\}$.

\begin{definition}
	$\{P,Q\}$ is called an admissible family of invariant sets with respect to $J$ at level $c$ provided that the following deformation property holds: if $K_{c}\backslash W=\emptyset$, then, there exists $\varepsilon_{0}>0$ such that for $\varepsilon\in(0,\varepsilon_{0})$, there exists $\eta\in C(X,X)$, satisfying\\
	$(1)\;\eta(\overline P)\subset\overline P,\eta(\overline Q)\subset\overline Q;$\\
	$(2)\;\eta|_{J^{c-2\varepsilon}}=id;$\\
	$(3)\;\eta(J^{c+\varepsilon}\backslash W)\subset J^{c- \varepsilon}.$
\end{definition}

Now we introduce a critical point theorem (see \cite[theorem 2.4]{liu2015multiple}).
\newtheorem{thA}{Theorem}[section]
\begin{thA}\label{theorem3.1}
	Assume that $\{P,Q\}$  is an admissible family of invariant sets with respect to $J$ at any level $c\geqslant c_{*}:=\inf_{u\in \Sigma}J(u)$ and there exists a map $\varphi_{0}:\Delta\to X$ satisfying\\
	$(1)$\;$\varphi_{0}(\partial_{1}\Delta)\subset P$ and $\varphi_{0}(\partial_{2}\Delta)\subset Q,$\\
	$(2)$\;$\varphi_{0}(\partial_{0}\Delta)\cap M=\emptyset,$\\
	$(3)$\;$\sup_{u\in\varphi_{0}(\partial_{0}\Delta)}J(u)<c_{*},$\\
	where $\Delta=\{(t_{1},t_{2})\in\mathbb{R}^{2}:t_{1},t_{2}\geqslant0,t_{1}+t_{2}\leqslant1\},\partial_{1}\Delta=\{0\}\times[0,1],\partial_{2}\Delta=[0,1]\times\{0\}$ and $\partial_{0}\Delta=\{(t_{1},t_{2})\in\mathbb{R}^{2}:t_{1},t_{2}\geqslant0,t_{1}+t_{2}=1\}$. Define\\
	\[c=\inf_{\varphi\in\Gamma}\sup_{u\in\varphi(\Delta)\backslash W}J(u),\]
	where $\Gamma:=\{\varphi\in C(\Delta,X):\varphi(\partial_{1}\Delta)\subset P,\varphi(\partial_{2}\Delta)\subset Q,\varphi|_{\partial_{0}\Delta}=\varphi_{0}|_{\partial_{0}\Delta}\}$. Then $c\geqslant c_{*}$ is a critical value of $J$ and $K_{c}\backslash W\not=\emptyset$.
\end{thA}

Now, we use Theorem \ref{theorem3.1} to prove the existence of a sign-changing solution to problem \eqref{2.1}, and for this, we take $X=E, P=P_{\varepsilon}^{+}, Q=P_{\varepsilon}^{-}$, and $J = I$. We will show that $\{P_{\varepsilon}^{+},P_{\varepsilon}^{-}\}$ is an admissible family of invariant sets for the functional $I$ at any level $c\in R$. Indeed, if $K_{c}\backslash W=\emptyset$, then $K_{c}\subset W$. Since $\mu>2p$ and $I$ satisfies the (PS)-condition, therefore $K_{c}$ is compact. Thus, $2\delta:=\text{dist}(K_{c},\partial W)>0$.

\begin{lem}\label{lemma3.8}
	If $K_{c}\backslash W=\emptyset$, there exists $\varepsilon_{0}>0$, such that for $0<\varepsilon<\varepsilon'<\varepsilon_{0}$, there exists a continuous
	map $\sigma:[0,1]\times E\to E$ satisfying\\
	$(1)$\;$\sigma(0,u)=u$ for $u\in E.$\\
	$(2)$\;$u\notin I^{-1}[c-\varepsilon',c+\varepsilon']$ for $t\in[0,1], \sigma(t,u)=u;$\\
	$(3)$\;$\sigma(1,I^{c+\varepsilon}\backslash W)\subset I^{c-\varepsilon};$\\
	$(4)$\;$\sigma(t,\overline{P_{\varepsilon}^{+}})\subset \overline{P_{\varepsilon}^{+}},\sigma(t,\overline{P_{\varepsilon}^{-}})\subset \overline{P_{\varepsilon}^{-}}$ for $t\in[0,1].$
\end{lem}
\begin{proof}
	For $G\subset E$ and $a>0$, let $N_{a}(G):=\{u\in E,\text{dist}(u,G)<a\}$. Since $K_{c}\backslash W=\emptyset$,  $K_{c}\subset W$. $W$ is an open set, then $N_{\delta}(K_{c})\subset W$. By Lemma \ref{lemma3.4}, we prove that $I(u)$ satisfies the (PS) condition, there exist $\varepsilon_{0},\alpha>0$, such that
	\begin{equation}
		\|I'(u)\|_{E^{*}}\geqslant\alpha\;\;\text{for}\;\;u\in I^{-1}[c-\varepsilon_{0},c+\varepsilon_{0}]\backslash N_{\frac{\delta}{2}}(K_{c}).\label{3.19}\tag{3.19}
	\end{equation}
	By Lemma \ref{lemma3.5}, there exists $\beta>0$ such that
	\begin{equation}
		\|u-A(u)\|_{E}\geqslant\beta\;\;\text{for}\;\;u\in I^{-1}[c-\varepsilon_{0},c+\varepsilon_{0}]\backslash N_{\frac{\delta}{2}}(K_{c}).\label{3.20}\tag{3.20}
	\end{equation}
	Since $I\in C^{1}(E,\mathbb{R})$, then for $\tau(s,u)\in N_{\frac{\delta}{2}}(K_{c})$, there exists $L>0$ such that
\begin{equation}
	\|\tau(s,u)\|_{E}+\|A\bigl(\tau(s,u)\bigr)\|_{E}\leqslant L.\label{3.21}\tag{3.21}
\end{equation}
	Without loss of generality, assume that $0<\varepsilon_{0}\leqslant\min\{\frac{a_{1}\beta\delta L^{p-2}}{16},\frac{\beta^{p-1}a_{1}\delta}{8}\}$. Take a cut-off function $g:E\to[0,1]$, which is locally Lipschitz continuous, such that
	\[g(u)=
	\begin{cases}
		0,&\text{if}\;u\notin I^{-1}[c-\varepsilon',c+\varepsilon']\;\text{or}\;u\in N_{\frac{\delta}{4}}(K_{c}),\\
		1,&\text{if}\;u\in I^{-1}[c-\varepsilon,c+\varepsilon]\;\text{and}\;u\notin N_{\frac{\delta}{2}}(K_{c}).\label{3.22}\tag{3.22}
	\end{cases}
	\]
	Decreasing $\varepsilon_{0}$ if necessary, one may find a  $\nu>0$ such that $I^{-1}[c-\varepsilon_{0},c+\varepsilon_{0}]\cap N_{\nu}(K)\subset N_{\frac{\delta}{4}}(K_{c})$, and this can be seen as a consequence of the (PS) condition. Thus $g(u)=0$ for any $u\in N_{\nu}(K)$.
	
	For $u\in E$, we consider the following initial value problem
	\[
	\begin{cases}
		\frac{\mathrm{d}\tau}{\mathrm{d}t}=-g(\tau)\bigl(\tau-B(\tau)\bigr),\\
		\tau(0,u)=u.\label{3.23}\tag{3.23}
	\end{cases}
	\]
	By Lemma \ref{lemma3.4} and Lemma \ref{lemma3.7}, $g(\tau)\bigl(\tau-B(\tau)\bigr)$ is locally Lipschitz continuous on $E$. For any $u\in E$, by existence and uniqueness theorem, there exists a unique solution $\tau(\cdot,u)\in C(\mathbb{R}^{+},E)$ for equation \eqref{4.1}, and $\tau(t,u)$ is uniformly continuous for $u$ on $[0,1]$. Define $\sigma(t,u)=\tau(\theta t,u)$, where $\theta=\max\{\frac{6\varepsilon L^{2-p}}{a_{1}\beta^{2}},\frac{6\varepsilon}{a_{1}\beta^{p}}\}$.
	
	By the proof of existence and uniqueness theorem, conclusion $(1)$ and conclusion $(2)$ are satisfied.
	We let $u\in I^{c+\varepsilon}\backslash W$. If $1<p\leqslant2$, without loss of generality, we assume $\tau(t,u)\not\equiv A\bigl(\tau(t,u)\bigr)$. By Lemma \ref{lemma3.7}, we have
	\begin{align}
		\frac{\mathrm{d}I\bigl(\tau(t,u)\bigr)}{\mathrm{d}t}&=\bigl\langle I'\bigl(\tau(t,u)\bigr),\tau'(t,u)\bigr\rangle\notag\\
		&=-g\bigl(\tau(t,u)\bigr)\bigl\langle I'\bigl(\tau(t,u)\bigr),\tau(t,u)-B\bigl(\tau(t,u)\bigr)\bigr\rangle\notag\\
		&\leqslant -\frac{1}{2}a_{1}g\bigl(\tau(t,u)\bigr)\|\tau(t,u)-A\bigl(\tau(t,u)\bigr)\|_{E}^{2}\bigl(\|\tau(t,u)\|_{E}+\|A\bigl(\tau(t,u)\bigr)\|_{E}\|\bigr)^{p-2}\notag\\
		&\leqslant 0.\notag
	\end{align}
	Thus we get that $I\bigl(\tau(t,u)\bigr)$ is decreasing for $t\geqslant0$. Moreover,
	\[I\bigl(\tau(t,u)\bigr)\leqslant I\bigl(\tau(0,u)\bigr)=I(u)\leqslant c+\varepsilon,\;\;\text{for}\;\;(t,u)\in[0,1]\times E.\]
	
	We claim that $I\bigl(\tau(t_{0},u)\bigr)<c-\varepsilon$ for any $t_{0}\in [0,\theta]$. Assume on the contrary that for any $t\in [0,\theta],I\bigl(\tau(t,u)\bigr)\geqslant c-\varepsilon$. Then $\tau(t,u)\in I^{-1}[c-\varepsilon,c+\varepsilon]$ for $t\in [0,\theta]$. We will show that for any $t\in [0,\theta],\tau(t,u)\notin N_{\frac{\delta}{2}}(K_{c})$. If, there exists $t_{1}\in [0,\theta]$, such that $\tau(t_{1},u)\in N_{\frac{\delta}{2}}(K_{c})$. Since $u\in I^{c+\varepsilon}\backslash W$ and $N_{\delta}(K_{c})\subset W$, $u\notin N_{\delta}(K_{c})$. By Lemma \ref{lemma3.7}, we obtain
	\begin{align}
		\frac{\delta}{2}&\leqslant\|\tau(t_{1},u)-u\|_{E}\leqslant\|\tau(t_{1},u)-\tau(0,u)\|_{E}=\|\int_{0}^{t_{1}}\tau'(s,u)\mathrm{d}s\|_{E}\leqslant\int_{0}^{t_{1}}\|\tau'(s,u)\|_{E}\mathrm{d}s\notag\\
		&=\int_{0}^{t_{1}}\|\tau(s,u)-B\bigl(\tau(s,u)\bigr)\|_{E}\mathrm{d}s\leqslant2\int_{0}^{t_{1}}\|\tau(s,u)-A\bigl(\tau(s,u)\bigr)\|_{E}\mathrm{d}s.\label{3.24}\tag{3.24}
	\end{align}
	If $1<p\leqslant2$, combining \eqref{3.20}, \eqref{3.21}, \eqref{3.24} and Lemma \ref{lemma3.7}, we have
	\begin{align}
		\frac{\delta}{2}&\leqslant\frac{2}{\beta} L^{2-p}\int_{0}^{t_{1}}\|\tau(s,u)-A\bigl(\tau(s,u)\bigr)\|_{E}^{2}\bigl(\|\tau(s,u)\|_{E}+\|A\bigl(\tau(s,u)\bigr)\|_{E}\bigr)^{p-2}\mathrm{d}s\notag\\
		&\leqslant\frac{4}{a_{1}\beta}L^{2-p}\int_{0}^{t_{1}}\langle I'\bigl(\tau(s,u)\bigr),\tau(s,u)-B\bigl(\tau(s,u)\bigr)\rangle \mathrm{d}s\notag\\
		&=\frac{4}{a_{1}\beta}L^{2-p}\biggl(I\bigl(\tau(t_{1},u)\bigr)-I\bigl(\tau(0,u)\bigr)\biggr)\notag\\
		&\leqslant\frac{8\varepsilon L^{2-p}}{a_{1}\beta}.\notag
	\end{align}
	Hence,
	\[\varepsilon\geqslant\frac{a_{1}\beta\delta L^{p-2}}{16}.\]
	This is a contradiction with $0<\varepsilon<\varepsilon_{0}\leqslant\min\{\frac{a_{1}\beta\delta L^{p-2}}{16},\frac{\beta^{p-1}a_{1}\delta}{8}\}$. Then for $t\in [0,\theta],\tau(t,u)\notin N_{\frac{\delta}{2}}(K_{c})$. Thus $g\bigl(\tau(s,u)\bigr)\equiv1$ for $t\in [0,\theta]$.
	Moreover, from Lemma \ref{lemma3.7}, if $1<p\leqslant2$, we obtain
	\begin{align}
		I\bigl(\tau(\theta,u)\bigr)&=I\bigl(\tau(0,u)\bigr)+\int_{0}^{\theta}\frac{\mathrm{d}I\bigl(\tau(s,u)\bigr)}{\mathrm{d}s}\mathrm{d}s\notag\\
		&=I(u)+\int_{0}^{\theta}\langle I'\bigl(\tau(s,u)\bigr),\tau'(s,u)\rangle\mathrm{d}s\notag\\
		&=I(u)-\int_{0}^{\theta}\langle I'\bigl(\tau(s,u)\bigr),\tau(s,u)-B\bigl(\tau(s,u)\bigr)\rangle\mathrm{d}s\notag\\
		&\leqslant I(u)-\int_{0}^{\theta}\frac{a_{1}}{2}\|\tau(s,u)-A\bigl(\tau(s,u)\bigr)\|_{E}^{2}\bigl(\|\tau(s,u)\|_{E}+\|A\bigl(\tau(s,u)\bigr)\|_{E}\bigr)^{p-2}\mathrm{d}s.\notag
	\end{align}
	Here we use \eqref{3.20} and \eqref{3.24}, 
	\begin{align}
		I\bigl(\tau(\theta,u)\bigr)&\leqslant c+\varepsilon-\theta\frac{a_{1}c_{1}}{2}\beta^{2}L^{p-2}\notag\\
		&\leqslant c-2\varepsilon.\notag
	\end{align}
	This is a contradiction with $I\bigl(\tau(t,u)\bigr)\geqslant c-\varepsilon$ for $t\in[0,\theta]$. Therefore, if $1<p\leqslant2$, there exists $t_{0}\in [0,\theta]$ such that $I\bigl(\tau(t_{0},u)\bigr)<c-\varepsilon$.\\
	If $p\geqslant2$, using Lemma \ref{lemma3.7}, one sees that
	\begin{align}
		\frac{\mathrm{d}I\bigl(\tau(t,u)\bigr)}{\mathrm{d}t}&=\bigl\langle I'\bigl(\tau(t,u)\bigr),\tau'(t,u)\bigr\rangle\notag\\
		&=-g\bigl(\tau(t,u)\bigr)\bigl\langle I'\bigl(\tau(t,u)\bigr),\tau(t,u)-B\bigl(\tau(t,u)\bigr)\bigr\rangle\notag\\
		&\leqslant -\frac{1}{2}a_{1}g\bigl(\tau(t,u)\bigr)\|\tau(t,u)-A\bigl(\tau(t,u)\bigr)\|_{E}^{p}\notag\\
		&\leqslant 0.\notag
	\end{align}
	Similarly, if $p\geqslant2$, using Lemma \ref{lemma3.7} again, we have
	\begin{align}
		\frac{\delta}{2}&\leqslant\int_{0}^{t_{1}}\|\tau(s,u)-A\bigl(\tau(s,u)\bigr)\|_{E}^{1-p}\|\tau(s,u)-A\bigl(\tau(s,u)\bigr)\|_{E}^{p}\mathrm{d}s\notag\\
		&\leqslant\beta^{1-p}\frac{2}{a_{1}}\int_{0}^{t_{1}}\langle I'\bigl(\tau(s,u)\bigr),\tau(s,u)-B\bigl(\tau(s,u)\bigr)\rangle \mathrm{d}s\notag\\
		&\leqslant\beta^{1-p}\frac{2}{a_{1}}2\varepsilon,\notag
	\end{align}
	which further implies that
	\[\varepsilon\geqslant\frac{a_{1}\beta^{p-1}\delta}{8}.\]
	Similarly, it shows a contradiction. Then for $t\in [0,\theta],\tau(t,u)\notin N_{\frac{\delta}{2}}(K_{c})$. Therefore, $g\bigl(\tau(s,u)\bigr)\equiv1$ for $t\in [0,\theta]$.
	If $p\geqslant2$, by Lemma \ref{lemma3.7}, we get
	\begin{align}
		I\bigl(\tau(\theta,u)\bigr)&=I\bigl(\tau(o,u)\bigr)+\int_{0}^{\theta}\frac{dI\bigl(\tau(s,u)\bigr)}{\mathrm{d}s}\mathrm{d}s\notag\\
		&=I(u)+\int_{0}^{\theta}\langle I'\bigl(\tau(s,u)\bigr),\tau'(s,u)\rangle \mathrm{d}s\notag\\
		&=I(u)-\int_{0}^{\theta}\langle I'\bigl(\tau(s,u)\bigr),\tau(s,u)-B\bigl(\tau(s,u)\bigr)\rangle \mathrm{d}s\notag\\
		&\leqslant I(u)-\int_{0}^{\theta}\frac{a_{1}}{2}\|\tau(s,u)-A\bigl(\tau(s,u)\bigr)\|_{E}^{p}\mathrm{d}s\notag\\
		&\leqslant c+\varepsilon-\theta\frac{a_{1}}{2}\beta^{p}\notag\\
		&\leqslant c-2\varepsilon.\notag
	\end{align}
	This is a contradiction with $I\bigl(\tau(t,u)\bigr)\geqslant c-\varepsilon$ for $t\in[0,\theta]$. Therefore, if $p\geqslant2$, there exists $t_{0}\in [0,\theta]$ such that $I\bigl(\tau(t_{0},u)\bigr)<c-\varepsilon$.  Furthermore, the conclusion $(3)$ is satisfied. Arguing as the
	proof of Lemma 2.1 in \cite{liu2001invariant}, the conclusion $(4)$ is a consequence of $(1)$ of Lemma \ref{lemma3.2}.
\end{proof}

\begin{corollary}
	$\{P_{\varepsilon}^{+},P_{\varepsilon}^{-}\}$ is an admissible family of invariant sets for the functional $I$ at any
	level $c\in\mathbb{R}$. 
\end{corollary}
\begin{proof}
	The conclusion follows from Lemma \ref{lemma3.8}.
\end{proof}

\begin{lem}\label{lemma3.9}
	For $q\in [2,6]$, there exists $m_{q}>0$ independent of $\varepsilon$ such that $\|u\|_{q}\leqslant m_{q}\varepsilon$ for
	$u\in M=P_{\varepsilon}^{+}\cap P_{\varepsilon}^{-}$.
\end{lem}
\begin{proof}
	This follows from the proof of Lemma \ref{lemma3.6}. 
\end{proof}

\begin{lem}\label{lemma3.10}
	If $\varepsilon>0$ is small enough, then $I(u)\geqslant\frac{1}{p}\varepsilon^{p}$ for $u\in \Sigma=\partial P_{\varepsilon}^{+}\cap\partial P_{\varepsilon}^{-}$, that is, $c_{*}\geqslant\frac{1}{p}\varepsilon^{p}$.
\end{lem}
\begin{proof}
	For any fixed $u\in \Sigma=\partial P_{\varepsilon}^{+}\cap\partial P_{\varepsilon}^{-}$, we have $\|u^{\pm}\|_{E}\geqslant\text{dist}(u,P^{\mp})=\varepsilon$. By $(f_{1})-(f_{2})$, one sees that
	\[F(t)\leqslant \frac{1}{2pm_{p}^{p}}|t|^{p}+C|t|^{q},\qquad\forall t\in\mathbb{R}.\label{3.25}\tag{3.25}\]
	Note that
	\begin{align}
		I(u)=\frac{1}{p}\bigl(\|u^{+}\|_{E}^{p}+\|u^{-}\|_{E}^{p}\bigr)+\frac{1}{2p}\int_{\mathbb{R}^{3}}\phi_{u}|u|^{p}-\int_{\mathbb{R}^{3}}F(u).\label{3.26}\tag{3.26}
	\end{align}
	Combining $q>p$ and \eqref{3.25}, we obtin
	\begin{align}
		I(u)&\geqslant \frac{1}{p}\bigl(\|u^{+}\|_{E}^{p}+\|u^{-}\|_{E}^{p}\bigr)-\int_{\mathbb{R}^{3}}\biggl(\frac{1}{2pm_{p}^{p}}|u|^{p}+C|u|^{q}\biggr)\notag\\
		&\geqslant \frac{2}{p}\varepsilon^{p}-\frac{1}{2pm_{p}^{p}}m_{p}^{p}\varepsilon^{p}-C m_{q}^{q}\varepsilon^{q}\notag\\
		&\geqslant\frac{1}{p}\varepsilon^{p}.\label{3.27}\tag{3.27}
	\end{align}
\end{proof}
Next we will prove Theorem \ref{theorem1.1}.
\begin{proof}
	\textbf{(Existence part)} In the following, we will construct $\varphi_{0}$ satisfying the hypotheses in Theorem \ref{theorem3.1}. Choose $v_{1},v_{2}\in C_{0}^{\infty}(\mathbb{R}^{3})\setminus\{0\}$, satisfying $\text{supp}(v_{1})\cap \text{supp}(v_{2})=\emptyset$ and $v_{1}\leqslant 0,v_{2}\geqslant 0$. For $(t,s)\in\Delta$, Let $\varphi_{0}(t,s):=R(tv_{1}+sv_{2})$, where $R$  is a positive constant to be determined later. Obviously, for $t,s\in[0,1],\varphi_{0}(0,s)=Rsv_{2}\in P_{\varepsilon}^{+}$ and $\varphi_{0}(t,0)=Rtv_{1}\in P_{\varepsilon}^{-}$. Therefore, the condition $(1)$ in Theorem \ref{theorem3.1} is satisfied.
	
	Now we show that $\varphi_{0}(\partial_{0}\Delta)\cap M=\emptyset$ for $R$ large enough.
	Let $\rho=\min\{\|tv_{1}+(1-t)v_{2}\|_{p}:0\leqslant t\leqslant1\}>0$, for all $u\in\varphi_{0}(\partial_{0}\Delta)$,
	\[\|u\|_{p}=R\|tv_{1}+(1-t)v_{2}\|_{p}\geqslant R\min \|tv_{1}+(1-t)v_{2}\|_{p}=\rho R.\]
	If $\varphi_{0}(\partial_{0}\Delta)\cap M\not=\emptyset$, we choose $u\in \varphi_{0}(\partial_{0}\Delta)\cap M$. By Lemma \ref{lemma3.9}, there exists $\|u\|_{p}\leqslant m_{p}\varepsilon$ for $\varepsilon$ small enough, however we have $\|u\|_{p}\geqslant \rho R$ for $R$ large enough, which is a contradiction. Therefore, the condition $(2)$ in Theorem \ref{theorem3.1} is satisfied.
	
	By $(f_{3})$, we have $F(t)\geqslant C_{1}|t|^{\mu}-C_{2}$ for any $t\in \mathbb{R}$. Notice that for any $u\in \varphi_{0}(\partial_{0}\Delta)$, we get 
	\begin{align}
		I(u)=&\frac{1}{p}\|u\|_{E}^{p}+\frac{1}{2p}\int_{\mathbb{R}^{3}}\phi_{u}|u|^{p}-\int_{\text{supp}(v_{1})\cap\text{supp}(v_{2})}F(u)\notag\\
		\leqslant&\frac{1}{p}\|u\|_{E}^{p}+\frac{C}{2p}\|u\|_{E}^{2p}-C_{1}\|u\|_{\mu}^{\mu}+C_{3}\notag\\
		=&\frac{1}{p}R^{p}\|tv_{1}+(1-t)v_{2}\|_{E}^{p}+\frac{C}{2p}R^{2p}\|tv_{1}+(1-t)v_{2}\|_{E}^{2p}\notag\\
		&-C_{1}R^{\mu}\|tv_{1}+(1-t)v_{2}\|_{\mu}^{\mu}+C_{3}.\notag
	\end{align}
	Combining Lemma \ref{lemma3.10} and $\mu>2p$, for $R$ largh enough and $\varepsilon$ small enough, we obtain
	\[\sup_{u\in\varphi_{0}(\partial_{0}\Delta)}I(u)<0<c_{*}.\]
	Therefore, the condition $(3)$ in Theorem \ref{theorem3.1} is satisfied. Arguing as in Theorem \ref{theorem3.1},\;$c$ is critical value, there exists at least one critical point $u$ on $E\backslash(P_{\varepsilon}^{+}\cup P_{\varepsilon}^{-})$  with respect to $I$, and it also is a sign-changing solution to equation \eqref{3.1}. Therefore $(u,\phi_{u})$ is a sign-changing solution of system \eqref{1.1}.
\end{proof}

\subsection{Existence of infinitely many sign-changing solutions}

In this section, we prove the existence of infinitely many sign-changing solutions to system \eqref{1.1}. For this, we will make use of \cite[Theorem 2.5]{liu2015multiple} , which we recall below. Assume $G: X\to X$ to be an isometric involution, that is, $G^{2}=id$ and $d(Gx,Gy)=d(x,y)$ for $x,y\in X$. We assume $J$ is $G$-invariant on $X$ in the sense that $J(Gx)=J(x)$ for any $x\in X$. We also assume $Q=GP$. A subset $F\subset X$ is said to be symmetric if $Gx\in F$ for any $x\in F$. The genus of a closed symmetric subset $F$ of $X\backslash\{0\}$ is denoted by $\gamma(F)$.

\begin{definition}
	$P$ is called a $G$-admissible invariant set with respect to $J$ at level $c$, if the
	following deformation property holds, there exist $\varepsilon_{0}>0$ and a symmetric open neighborhood $N$ of $K_{c} \backslash W$ with $\gamma(\overline{N})<\infty$, such that for $\varepsilon \in\left(0,\varepsilon_{0}\right)$, there exists $\eta \in C(X, X)$ satisfying\\
	$(1)\;\eta(\overline{P}) \subset \overline{P}, \eta(\overline{Q}) \subset \overline{Q};$\\
	$(2)\;\eta \circ G=G \circ \eta;$\\
	$(3)\;\eta|_{J^{c-2 \varepsilon}}=id;$\\
	$(4)\;\eta\left(J^{c+\varepsilon} \backslash(N \cup W)\right) \subset J^{c-\varepsilon}.$\\
	
\end{definition}

\begin{thA}\label{theorem3.2}
	Assume that $P$ is a $G-$admissible invariant set with respect to $J$  at any
	level $c\geqslant c^{*}:=\inf_{u \in \Sigma} J(u)$  and for any $n \in \mathbb{N}$,  there exists a continuous map $\varphi_{n}: B_{n}:=$ $\left\{x \in \mathbb{R}^{n}:|x| \leqslant 1\right\} \rightarrow X$ satisfying\\
	$(1)$\;$\varphi_{n}(0) \in M:=P \cap Q, \varphi_{n}(-t)=G \varphi_{n}(t)$ for $t \in B_{n},$\\
	$(2)\;\varphi_{n}\left(\partial B_{n}\right) \cap M=\emptyset;$\\
	$(3)\;\sup _{u \in \operatorname{Fix}_{G} \cup \varphi_{n}\left(\partial B_{n}\right)} J(u)<c^{*},$ where $\ \operatorname{Fix}_{G}:=\{u \in X: G u=u\}$.\\
	For $j \in \mathbb{N},$ define
	$$
	c_{j}=\inf _{B \in \Gamma_{j}} \sup _{u \in B \backslash W} J(u),
	$$
	where
	$$
	\Gamma_{j}:=\left\{\begin{array}{l}
		B\;\bigg| \begin{array}{l}
			\ B=\varphi\left(B_{n} \backslash Y\right)\;\text {for some}\;\varphi \in G_{n}, \ n \geqslant j,
			\text {and open}\;Y \subset B_{n}\\
			\;\text {such that}-Y=Y \text {and}\;\gamma(\overline{Y}) \leqslant n-j
		\end{array}
	\end{array}\right\},
	$$
	and
	$$
	G_{n}:=\left\{\varphi \bigg| \begin{array}{l}
		\varphi \in C\left(B_{n}, X\right), \text{for}\;t \in B_{n},\ \varphi(-t)=G \varphi(t), \\
		\varphi(0) \in M \text {and}\left.\varphi\right|_{\partial B_{n}}=\left.\varphi_{n}\right|_{\partial B_{n}}
	\end{array}\right\}.
	$$
	Then for $j\geqslant2,c_{j}\geqslant c^{*}$ is a critical value of $J$,\;$K_{c_{j}}\backslash W\neq\emptyset$,\;$c_{j}\rightarrow\infty$ as $j\rightarrow\infty$.
\end{thA}

To apply Theorem \ref{theorem3.2}, we take $X=E, G=-id, J=I$ and $P=P_{\varepsilon}^{+}$. Then $M=P_{\varepsilon}^{+}\cap P_{\varepsilon}^{-},\Sigma=\partial P_{\varepsilon}^{+}\cap\partial P_{\varepsilon}^{-}$, and $W=P_{\varepsilon}^{+}\cup P_{\varepsilon}^{-}$. In this subsection, $f$ is assumed to be odd, and, as a consequence, $I$ is even. Now, we show that $P_{\varepsilon}^{+}$ is a $G$-admissible invariant set for the functional $I$ at any level $c$. Since $K_{c}$ is compact, there exists a symmetric open neighborhood $N$ of $K_{c}\backslash W$ such that $\gamma(N)<\infty$.

\begin{lem}
	There exists $\varepsilon_{0}>0$ such that for $0<\varepsilon<\varepsilon^{\prime}<\varepsilon_{0}$, there exists a continuous map $\sigma:[0,1] \times E \rightarrow E$ satisfying\\
	$(1)$\;$\sigma(0, u)=u$ for $u\in E.$\\
	$(2)$\;$\sigma(t, u)=u$ for $t \in[0,1],\;u \notin I^{-1}\left[c-\varepsilon^{\prime}, c+\varepsilon^{\prime}\right].$\\
	$(3)$\;$\sigma(t,-u)=-\sigma(t, u)$ for $(t, u) \in[0,1] \times E.$\\
	$(4)$ $\sigma\left(1, I^{c+\varepsilon} \backslash(N \cup W)\right) \subset I^{c-\varepsilon}.$\\
	$(5)$\;$\sigma\left(t,\overline{P_{\varepsilon}^{+}}\right)\subset \overline{P_{\varepsilon}^{+}},\,\sigma\left(t,\overline{P_{\varepsilon}^{-}}\right)\subset \overline{P_{\varepsilon}^{-}}$ for $t \in[0,1].$
\end{lem}
\begin{proof}
	The proof is similar to the proof of Lemma \ref{lemma3.8}.
\end{proof}

\begin{corollary}
	$P_{\varepsilon}^{+}$ is a $G$-admissible invariant set for the functional $I$ at any level $c$.
\end{corollary}

\begin{proof}
	\textbf{(Multiplicity part)}
	Referring to \cite{liu2015multiple}, we construct $\varphi_{n}$ satisfying the hypotheses of Theorem \ref{theorem3.2}. For any fixed $n\in\mathbb{N}$, we choose $\{v_{i}\}_{1}^{n}\subset C_{0}^{\infty}(\mathbb{R}^{3})\backslash \{0\}$ such that $\text{supp}(v_{i})\cap\text{supp}(v_{j})=\emptyset$ for $i\ne j$. Define $\varphi_{n}\in C(B_{n},E)$, as
	\[\varphi_{n}(t)=R_{n}\sum_{i=1}^{n}t_{i}v_{i},\qquad t=(t_{1},t_{2},\cdots,t_{n})\in B_{n},\]
	where $R_{n}>0$. For $R_{n}$ large enough, we can see that $\varphi_{n}(\partial B_{n})\cap(P_{\varepsilon}^{+}\cap P_{\varepsilon}^{-})=\emptyset$, which means the condition $(2)$ of Theorem \ref{theorem3.2} is satisfied.
	
	Moreover,
	\[\sup_{u\in\varphi_{n}(\partial B_{n})}I(u)<0<\inf_{u\in \Sigma}I(u)=c^{*},\]
	which means the condition $(3)$ of Theorem \ref{theorem3.2} is satisfied. Obviously, we obtain $\varphi_{0}(0)=0\in P_{\varepsilon}^{+}\cap P_{\varepsilon}^{-},\varphi_{n}(-t)=-\varphi_{n}(t)$ for $t\in B_{n}$. This means the condition $(1)$ of Theorem \ref{theorem3.2} is satisfied.
	
	Therefore, there exist infinitely many sign-changing solutions to equation \eqref{3.1} i.e. system \eqref{1.1}.
\end{proof}

\section{Proof of Theorem 1.2}

In this section, we assume $\mu>\frac{2p(p+1)}{p+2}$. Since we do not impose $\mu>2p$, the boundedness of the (PS)-sequence becomes not easy to establish. This obstacle will be overcome via a perturbation approach which is originally due to \cite{liu2015multiple}. The method from Section 3 can be used for the perturbed problem. By passing to the limit, we  obtain sign-changing solutions of the original problem \eqref{1.1}.

Fix a number $r\in\bigl(\text{max}(2p,q),p^{*}\bigr)$. For any fixed $\lambda\in(0,1]$, we consider the modified problem
\begin{equation}
	-\Delta_{p}u+V(x)|u|^{p-2}u+\phi_{u}|u|^{p-2}u=f(u)+\lambda|u|^{r-2}u,\qquad u\in E\label{4.1}\tag{4.1}
\end{equation}
and its associated functional
\[I_{\lambda}(u)=I(u)-\frac{\lambda}{r}\int_{\mathbb{R}^{3}}|u|^{r}.\]
It is standard to show that $I_{\lambda}\in C^{1}(E,\mathbb{R})$ and
\[\langle I_{\lambda}'(u),v\rangle=\langle I'(u),v\rangle-\lambda\int_{\mathbb{R}^{3}}|u|^{r-2}uv,\qquad u,v\in E.\]
For any $u\in E$, we denote by $v=A_{\lambda}(u)\in E$ the unique solution to the problem
\[-\Delta_{p}v+V(x)|v|^{p-2}v+\phi_{u}|v|^{p-2}v=f(u)+\lambda|u|^{r-2}u,\qquad v\in E.\label{4.2}\tag{4.2}\]
As in Section 3, one verifies that the operator $A_{\lambda}: E\to E$ is well defined and is continuous and compact. In the following, if the proof of a result is similar to its counterpart in Section 3, it will not be written out.

\begin{lem}\label{lemma4.1}
	$(1)$ There exist $1<p\leqslant2,a_{1}>0$ and $a_{2}>0$ such that
	\[\langle I_{\lambda}'(u),u-A_{\lambda}(u)\rangle\geqslant a_{1}\|u-A_{\lambda}(u)\|_{E}^{2}(\|u\|_{E}+\|A_{\lambda}(u)\|_{E})^{p-2}\]
	and
	\[\|I_{\lambda}'(u)\|_{E^{*}}\leqslant a_{2}\|u-A_{\lambda}(u)\|_{E}^{p-1}\bigl(1+\|u\|_{E}^{p}\bigr)\]
	hold for every $u\in E$.\\
	$(2)$There exist $p\geqslant2,a_{1}>0$ and $a_{2}>0$ such that
	\[\langle I_{\lambda}'(u),u-A_{\lambda}(u)\rangle\geqslant a_{1}\|u-A_{\lambda}(u)\|_{E}^{p}.\]
	and
	\[\|I_{\lambda}'(u)\|_{E^{*}}\leqslant a_{2}\|u-A_{\lambda}(u)\|_{E}\bigl(\|u\|_{E}+\|A_{\lambda}(u)\|_{E}\bigr)^{p-2}\bigl(1+\|u\|_{E}^{p}\bigr)\]
	hold for every $u\in E$.
\end{lem}
\begin{proof}
	The proof is similar to the proof of Lemma \ref{lemma3.3}, and we get
	\begin{align}
		&\langle I_{\lambda}'(u),u-A_{\lambda}(u)\rangle\notag\\
		=&\int_{\mathbb{R}^{3}}|\nabla u|^{p-2}\nabla u\nabla\bigl(u-A_{\lambda}(u)\bigr)+\int_{\mathbb{R}^{3}}V(x)|u|^{p-2}u\bigl(u-A_{\lambda}(u)\bigr)\notag\\
		&+\int_{\mathbb{R}^{3}}\phi_{u}|u|^{p-2}u\bigl(u-A_{\lambda}(u)\bigr)-\int_{\mathbb{R}^{3}}f(u)\bigl(u-A_{\lambda}(u)\bigr)-\lambda\int_{\mathbb{R}^{3}}|u|^{r-2}u(u-v)\notag\\
		=&\int_{\mathbb{R}^{3}}\bigl(|\nabla u|^{p-2}\nabla u-|\nabla v|^{p-2}\nabla v\bigr)(\nabla u-\nabla v)+\int_{\mathbb{R}^{3}}V(x)\bigl(|u|^{p-2}u-|v|^{p-2}v\bigr)(u-v)\notag\\
		&+\int_{\mathbb{R}^{3}}\phi_{u}\bigl(|u|^{p-2}u-|v|^{p-2}v\bigr)(u-v),\notag
	\end{align}
	which is the same as Lemma \ref{lemma3.3}, the proof is complete.
\end{proof}

\begin{lem}\label{lemma4.2}
	If $\{u_{n}\}\subset E$ is a bounded sequence with $I_{\lambda}'(u_{n})\to0$, then $\{u_{n}\}\subset E$  has a convergent subsequence.
\end{lem}

\begin{lem}\label{lemma4.3}
	For any $\lambda\in(0,1),a<b,\alpha >0$, there exists $\beta(\lambda)>0$ such that $\|u-A_{\lambda}(u)\|_{E}\geqslant\beta(\lambda)$ for any $u\in E$ with $I_{\lambda}(u)\in[a,b]$ and $\|I_{\lambda}'(u)\|\geqslant\alpha$.
\end{lem}
\begin{proof}
	By Lemma \ref{lemma4.2}, we prove that $I$ satisfies the (PS) condition. Assuming that $c\in\mathbb{R},n\in\mathbb{N}$, for $\{u_{n}\}\subset E$ with $I_{\lambda}(u_{n})\to c,I_{\lambda}'(u_{n})\to 0$ as $n\to\infty$. Fix a number $\gamma\in(2p,r)$, for $u\in E$,
	\begin{align}
		I_{\lambda}(u)&-\frac{1}{\gamma}\langle I'_{\lambda}(u),u\rangle\notag\\=&\;\frac{1}{p}\int_{\mathbb{R}^{3}}\bigl(|\nabla u|^{p}+V(x)|u|^{p}\bigr)+\frac{1}{2p}\int_{\mathbb{R}^{3}}\phi_{u}|u|^{p}-\int_{\mathbb{R}^{3}}F(u)-\frac{\lambda}{r}\int_{\mathbb{R}^{3}}|u|^{r}\notag\\
		&-\frac{1}{\gamma}\biggl(\int_{\mathbb{R}^{3}}|\nabla u|^{p}+\int_{\mathbb{R}^{3}}V(x)|u|^{p}+\int_{\mathbb{R}^{3}}\phi_{u}|u|^{p}-\int_{\mathbb{R}^{3}}f(u)u-\lambda\int_{\mathbb{R}^{3}}|u|^{r}\biggr)\notag\\
		=&\;\biggl(\frac{1}{p}-\frac{1}{\gamma}\biggr)\|u\|_{E}^{p}+\biggl(\frac{1}{2p}-\frac{1}{\gamma}\biggr)\int_{\mathbb{R}^{3}}\phi_{u}|u|^{p}+\int_{\mathbb{R}^{3}}\biggl(\frac{1}{\gamma}f(u)u-F(u)\biggr)\notag\\
		&+\lambda\biggl(\frac{1}{\gamma}-\frac{1}{r}\biggr)\int_{\mathbb{R}^{3}}|u|^{r}.\notag
	\end{align}
	Then Young's inequality implies that
	\begin{equation}
		\|u_{n}\|_{E}^{p}+\int_{\mathbb{R}^{3}}\phi_{u_{n}}|u_{n}|^{p}+\lambda\|u_{n}\|_{r}^{r}\leqslant d_{1}\bigl(\|I'_{\lambda}(u_{n})\|_{E^{*}}+\|I'_{\lambda}(u_{n})\|_{E^{*}}\|u_{n}\|_{E}+\|u_{n}\|_{q}^{q}\bigr).\label{4.3}\tag{4.3}
	\end{equation}
	Assume that there exist $u_{n}\subset E,I_{\lambda}(u_{n})\in[a,b]$ and $\|I_{\lambda}'(u_{n})\|_{E^{*}}\geqslant\alpha$, $\|u_{n}-A_{\lambda}(u_{n})\|_{E}\to0$ as $n\to\infty$. By \eqref{4.3}, for $n$ large enough, we get
	\[\|u_{n}\|_{E}^{p}+\lambda\|u_{n}\|_{r}^{r}\leqslant d_{2}\bigl(1+\|u_{n}\|_{q}^{q}\bigr).\]
	
	We claim that $\{u_{n}\}$ is bounded in $E$. Otherwise, assume that $u_{n}\to\infty$ as $n\to\infty$. Let $w_{n}=\frac{u_{n}}{\|u_{n}\|_{E}}$, passing to a subsequence, there exists $w\in E$ such that  $w_{n}\rightharpoonup w$ in $E$, by Remark \ref{remark2.2}, $w_{n}\to w$ in $L^{q},q\in \left[\;p,p^{*}\right)$. By \eqref{4.3} and $p<q<r$, we have
	\[\|u_{n}\|_{E}^{p-r}+\lambda\|w_{n}\|_{r}^{r}\leqslant \frac{d_{2}\bigl(1+\|u_{n}\|_{q}^{q}\bigr)}{\|u_{n}\|_{E}^{r}}.\]
	Hence,
	\[\|w_{n}\|_{r}^{r}\leqslant0.\]
	Therefore $w=0$. But if $n$ is large enough, there exists
	\begin{equation}
		\|u_{n}\|_{E}^{p}+\lambda\|u_{n}\|_{r}^{r}\leqslant d_{3}\|u_{n}\|_{q}^{q}.\label{4.4}\tag{4.4}
	\end{equation}
	By \eqref{4.4}, there exists a number $d(\lambda)>0$ such that for $n$ large enough, such that
	\[\|u_{n}\|_{p}^{p}+\|u_{n}\|_{r}^{r}\leqslant d(\lambda)\|u_{n}\|_{q}^{q}.\]
	Let $t\in (0,1)$ be such that
	\[\frac{1}{p}=\frac{t}{2}+\frac{1-t}{r}.\]
	Then the interpolation inequality implies that
	\[\|u_{n}\|_{p}^{p}+\|u_{n}\|_{r}^{r}\leqslant d(\lambda)\|u_{n}\|_{q}^{q}\leqslant d(\lambda)\|u_{n}\|_{p}^{tq}\|u_{n}\|_{r}^{(1-t)q}.\]
	Hence there exists $d_{1}(\lambda),d_{2}(\lambda)>0$, for $n$ large enough, such that
	\[d_{1}(\lambda)\|u_{n}\|_{p}^{\frac{p}{r}}\leqslant\|u_{n}\|_{r}\leqslant d_{2}(\lambda)\|u_{n}\|_{p}^{\frac{p}{r}}.\]
	Hence,
	\[\|u_{n}\|_{q}^{q}\leqslant d_{3}(\lambda)\|u_{n}\|_{p}^{p}.\]
	And by \eqref{4.4}, we have
	\[\|u_{n}\|_{p}^{p}+\|u_{n}\|_{r}^{r}\leqslant d_{4}(\lambda)\|u_{n}\|_{q}^{q}.\]
	Furthermore,
	\[\|w_{n}\|_{p}^{p}\geqslant \bigl(d_{5}(\lambda)\bigr)^{-1}.\]
	Since $w_{n}\to w$ in $L^{p}$, one sees that
	\[\|w\|_{p}^{p}\geqslant \bigl(d_{5}(\lambda)\bigr)^{-1},\]
	which contradicts $w=0$. Therefore $\{u_{n}\}$ is bounded in $E$. Similar to Lemma \ref{lemma3.5}, the proof is completed.
\end{proof}

\begin{lem}\label{lemma4.4}
	There exists $\varepsilon_{1}>0$ such that for $\varepsilon\in(0,\varepsilon_{1})$,\\
	$(1)$\;$A_{\lambda}(\partial P_{\varepsilon}^{-})\subset P_{\varepsilon}^{-}$ and every nontrivial solution $u\in P_{\varepsilon}^{-}$ is negative,\\
	$(2)$\;$A_{\lambda}(\partial P_{\varepsilon}^{+})\subset P_{\varepsilon}^{+}$ and every nontrivial solution $u\in P_{\varepsilon}^{+}$ is positive.
\end{lem}
\begin{proof}
	The proof is similar to the proof of Lemma \ref{lemma3.6}, and we get
	\begin{align}
		\text{dist}(v,P^{-})^{p-1}\|v^{+}\|_{E}\leqslant&\int_{\mathbb{R}^{3}}\bigl(\delta|u^{+}|^{p-1}+C_{\delta}|u^{+}|^{q-1}\bigr)v^{+}+\lambda\int_{\mathbb{R}^{3}}|u^{+}|^{r-2}u^{+}v^{+}\notag\\
		=&\delta\|u^{+}\|_{p}^{p-1}\|v^{+}\|_{p}+C_{\delta}\|u^{+}\|_{q}^{q-1}\|v^{+}\|_{q}+\lambda\|u^{+}\|_{r}^{r-1}\|v^{+}\|_{r}\notag\\
		\leqslant& C\bigl(\delta+C_{\delta}\text{dist}(u,P^{-})^{q-p}+\lambda\text{dist}(u,P^{-})^{r-p}\bigr)\notag\\
		&\text{dist}(u,P^{-})^{p-1}\|v^{+}\|_{E},\notag
	\end{align}
	which further implies that
	\[\text{dist}(v,P^{-})^{p-1}\leqslant C\bigl(\delta+C_{\delta}\text{dist}(u,P^{-})^{q-p}+\lambda\text{dist}(u,P^{-})^{r-p}\bigr)\text{dist}(u,P^{-})^{p-1}.\]
	There exists $\varepsilon_{1}>0$ such that for $\varepsilon\in(0,\varepsilon_{1})$,
	\[\text{dist}(v,P^{-})^{p-1}\leqslant\biggl(\frac{1}{6}+\frac{1}{6}+\frac{1}{6}\biggr)\text{dist}(u,P^{-})^{p-1},\qquad\forall u\in P_{\varepsilon}^{-}.\]
	Therefore, we have $A_{\lambda}(\partial P_{\varepsilon}^{-})\subset P_{\varepsilon}^{-}$. If there exists $u\in P_{\varepsilon}^{-}$ such that $A_{\lambda}(u)=u$, then $u\in P^{-}$. If $u\not\equiv0$, by the maximum principle, $u<0$ in $\mathbb{R}^{3}$.
\end{proof}

\begin{lem}\label{lemma4.5}
	There exists a locally Lipschitz continuous map $B_{\lambda}:E\backslash K_{\lambda}\to E$, where $K_{\lambda}:=Fix(A_{\lambda}),$ such that\\
	$(1)$\;$B_{\lambda}(\partial P_{\varepsilon}^{+})\subset P_{\varepsilon}^{+},\;B_{\lambda}(\partial P_{\varepsilon}^{-})\subset P_{\varepsilon}^{-}$ for $\varepsilon\in (0,\varepsilon_{1});$\\
	$(2)$\;for all $u\in E\backslash K_{\lambda},$
	\[\frac{1}{2}\|u-B_{\lambda}(u)\|_{E}\leqslant\|u-A_{\lambda}(u)\|_{E}\leqslant2\|u-B_{\lambda}(u)\|_{E};\]
	$(3)$\;for all $u\in E\backslash K_{\lambda},$ if $1<p\leqslant2$
	\[\langle I_{\lambda}'(u),u-B_{\lambda}(u)\rangle\geqslant\frac{1}{2}a_{1}\|u-A_{\lambda}(u)\|_{E}^{2}\bigl(\|u\|_{E}+\|A_{\lambda}(u)\|_{E}\bigr)^{p-2};\]
	if $p\geqslant2$,
	\[\langle I_{\lambda}'(u),u-B_{\lambda}(u)\rangle\geqslant\frac{1}{2}a_{1}\|u-A_{\lambda}(u)\|_{E}^{p};\]
	$(4)$\;if $I$ is even,\,$A$ is odd, then $B$ is odd.
\end{lem}
\begin{proof}
	The proof is similar to the proof of Lemma \ref{lemma3.7}.
\end{proof}

We are ready to prove Theorem \ref{theorem1.2}.
\begin{proof}
	\textbf{(Existence part)} \textbf{Step1.} In what follows, we divide two steps to complete the proof. We firstly claim that $\{P_{\varepsilon}^{+},P_{\varepsilon}^{-}\}$ is an admissible family of invariant sets for the functional $I_{\lambda}$ at any level $c$. Choose $v_{1},v_{2}\in C_{0}^{\infty}(B_{1}(0))\backslash\{0\},\;v_{1}\leqslant0,v_{2}\geqslant0$ such that $\text{supp}(v_{1})\cap \text{supp}(v_{2})=\emptyset$. For $(t,s)\in\Delta$, let
	\[\varphi_{0}(t,s)(\cdot):=R^{2}\bigl(tv_{1}(R\cdot)+sv_{2}(R\cdot)\bigr).\]
	where $R$ is a positive constant. Then for $t,s\in [0,1],\varphi_{0}(0,s)(\cdot)=R^{2}sv_{2}(R\cdot)\in P_{\varepsilon}^{+}$ and $\varphi_{0}(t,0)(\cdot)=R^{2}tv_{1}(R\cdot)\in P_{\varepsilon}^{-}$. Similar to \eqref{3.10}, for any $\varepsilon>0$,
	\[I_{\lambda}(u)\geqslant I(u)\geqslant\frac{1}{p}\varepsilon^{p},\qquad u\in\Sigma:=\partial P_{\varepsilon}^{+}\cap\partial P_{\varepsilon}^{-},\;\lambda\in(0,1).\]
	That means $c_{\lambda}^{*}:=\inf_{u \in \Sigma}I_{\lambda}(u)\geqslant\frac{1}{p}\varepsilon^{p}$ for $\lambda\in(0,1)$. Let $u_{t}=\varphi_{0}(t,1-t)$ for $t\in[0,1]$, we get
	\begin{align}
		\int_{\mathbb{R}^{3}}|\nabla u_{t}|^{p}&=R^{3p-3}\int_{\mathbb{R}^{3}}\bigl(t^{p}|\nabla v_{1}|^{p}+(1-t)^{p}|\nabla v_{2}|^{p}\bigr),\label{4.5}\tag{4.5}\\
		\int_{\mathbb{R}^{3}}|u_{t}|^{p}&=R^{2p-3}\int_{\mathbb{R}^{3}}\bigl(t^{p}|v_{1}|^{p}+(1-t)^{p}|v_{2}|^{p}\bigr),\label{4.6}\tag{4.6}\\
		\int_{\mathbb{R}^{3}}|u_{t}|^{\mu}&=R^{2\mu-3}\int_{\mathbb{R}^{3}}\bigl(t^{\mu}|v_{1}|^{\mu}+(1-t)^{\mu}|v_{2}|^{\mu}\bigr).\label{4.7}\tag{4.7}\\
		\int_{\mathbb{R}^{3}}\phi_{u_{t}}|u_{t}|^{p}&=R^{3p-3}\int_{\mathbb{R}^{3}}\phi_{\widetilde{u_{t}}}|\widetilde{u_{t}}|^{p},\;where\;\widetilde{u_{t}}=tv_{1}+(1-t)v_{2}.\label{4.8}\tag{4.8}
	\end{align}
	We use $(f_{3})$, $F(t)\geqslant C_{1}|t|^{\mu}-C_{2}$ for any $t\in\mathbb{R}$, and \eqref{4.5}-\eqref{4.8}. It follows that
	\begin{align}
		I_{\lambda}(u_{t})=&\;\frac{1}{p}\int_{\mathbb{R}^{3}}\bigl(|\nabla u_{t}|^{p}+V(x)|u_{t}|^{p}\bigr)+\frac{1}{2p}\int_{\mathbb{R}^{3}}\phi_{u_{t}}|u_{t}|^{p}\notag\\
		&-\int_{\text{supp}(v_{1})\cup \text{supp}(v_{2})}F(u_{t})-\frac{\lambda}{r}\int_{\mathbb{R}^{3}}|u_{t}|^{r}\notag\\
		\leqslant&\;\frac{1}{p}\int_{\mathbb{R}^{3}}\bigl(|\nabla u_{t}|^{p}+V(x)|u_{t}|^{p}\bigr)+\frac{1}{2p}\int_{\mathbb{R}^{3}}\phi_{u_{t}}|u_{t}|^{p}\notag\\
		&-C_{1}\|u_{t}\|_{\mu}^{\mu}+C_{3}-\frac{\lambda}{r}\int_{\mathbb{R}^{3}}|u_{t}|^{r}\notag\\
		\leqslant&\;\frac{R^{3p-3}}{p}\int_{\mathbb{R}^{3}}\bigl(t^{p}|\nabla v_{1}|^{p}+(1-t)^{p}|\nabla v_{2}|^{p}\bigr)+\frac{R^{3p-3}}{2p}\int_{\mathbb{R}^{3}}\phi_{\widetilde{u_{t}}}|\widetilde{u_{t}}|^{p}\notag\\
		&+\frac{R^{2p-3}}{p}\max_{|x|\leqslant1}V(x)\int_{\mathbb{R}^{3}}\bigl(t^{p}|v_{1}|^{p}+(1-t)^{p}|v_{2}|^{p}\bigr)\notag\\
		&-C_{1}R^{2\mu-3}\int_{\mathbb{R}^{3}}\bigl(t^{\mu}|v_{1}|^{\mu}+(1-t)^{\mu}|v_{2}|^{\mu}\bigr)+C_{2}R^{-3}\notag\\
		&-\frac{\lambda R^{2r-3}}{r}\int_{\mathbb{R}^{3}}\bigl(t^{\mu}|v_{1}|^{\mu}+(1-t)^{\mu}|v_{2}|^{\mu}\bigr).\notag
	\end{align}
	Since	  $\mu>p,r\in\bigl(\text{max}(2p,q),p^{*}\bigr)$, $I_{\lambda}(u)\to-\infty$ uniformly as $R\to \infty$ for $\lambda\in(0,1),t\in[0,1]$. Hence, choosing a $R$ independent of $\lambda$ and large enough, we have
	\begin{equation}
		\sup_{u\in\varphi_{0}(\partial_{0}\Delta)}I_{\lambda}(u)<0<c_{\lambda}^{*}:=\inf_{u \in \Sigma}I_{\lambda}(u),\qquad \lambda\in(0,1).\label{4.9}\tag{4.9}
	\end{equation}
	Since $\|u_{t}\|_{p}\to\infty$ as $R\to\infty$ for $t\in[0,1]$. It follows from the proof of Theorem \ref{theorem1.1}, we get $\varphi_{0}(\partial_{0}\Delta)\cap M=\emptyset$ for $R$ independent of $\lambda$ and large enough. Thus, the condition $(2)$ of Theorem \ref{theorem3.2} is satisfied. Therefore,
	\[c_{\lambda}=\inf_{\varphi\in\Gamma}\sup_{u\in\varphi(\Delta)\backslash W}I_{\lambda}(u),\]
	$c_{\lambda}$ is a critical value of $I_{\lambda}$ satisfing $c_{\lambda}\geqslant c_{\lambda}^{*}$. Thus, there exists $u_{\lambda}\in E\backslash(P_{\varepsilon}^{+}\cup P_{\varepsilon}^{-})$ such that $I_{\lambda}(u_{\lambda})=c_{\lambda},I_{\lambda}'(u_{\lambda})=0$.
	
	\textbf{Step2.} Passing to the limit as $\lambda\to0$. By the definition of $c_{\lambda}$, we have for $\lambda\in(0,1)$,
	\[c_{\lambda}\leqslant c(R):=\sup_{u\in\varphi_{0}(\partial_{0}\Delta)}I(u)<\infty.\]
	We claim that $\{u_{\lambda}\}$ is bounded in $E$ for $\lambda\in(0,1)$. Firstly, we have
	\begin{equation}
		I_{\lambda}(u_{\lambda})=c_{\lambda}=\frac{1}{p}\int_{\mathbb{R}^{3}}\bigl(|\nabla u_{\lambda}|^{p}+V(x)|u_{\lambda}|^{p}\bigr)+\frac{1}{2p}\int_{\mathbb{R}^{3}}\phi_{u_{\lambda}}|u_{\lambda}|^{p}-\int_{\mathbb{R}^{3}}\biggl(F(u_{\lambda})+\frac{\lambda}{r}|u_{\lambda}|^{r}\biggr)\label{4.10}\tag{4.10}
	\end{equation}
	and
	\begin{equation}
		I_{\lambda}'(u_{\lambda})u_{\lambda}=0=\int_{\mathbb{R}^{3}}\bigl(|\nabla u_{\lambda}|^{p}+V(x)|u_{\lambda}|^{p}+\phi_{u_{\lambda}}|u_{\lambda}|^{p}-f(u_{\lambda})u_{\lambda}-\lambda|u_{\lambda}|^{r}\bigr).\label{4.11}\tag{4.11}
	\end{equation}
	Moreover, similar to \cite{d2004non}, we have the Pohoz\v{a}ev identity,
	\begin{align}
		&\frac{3-p}{p}\int_{\mathbb{R}^{3}}|\nabla u_{\lambda}|^{p}+\frac{3}{p}\int_{\mathbb{R}^{3}}V(x)|u_{\lambda}|^{p}+\frac{1}{p}\int_{\mathbb{R}^{3}}|u_{\lambda}|^{p}\nabla V(x)\cdot x\notag\\
		&\;\;\;+\frac{5}{2p}\int_{\mathbb{R}^{3}}\phi_{u_{\lambda}}|u_{\lambda}|^{p}-\int_{\mathbb{R}^{3}}\biggl(3F(u_{\lambda})+\frac{3\lambda}{r}|u_{\lambda}|^{r}\biggr)=0.\label{4.12}\tag{4.12}
	\end{align}
	Multiplying \eqref{4.10} by $\frac{5\mu}{2}-3p$, \eqref{4.11} by $-1$, \eqref{4.12} by $p-\frac{\mu}{2}$ and adding them up, we obtain
	\begin{align}
		&\frac{2\mu-2p-2p^{2}+p\mu}{2p}\int_{\mathbb{R}^{3}}|\nabla u_{\lambda}|^{p}+\frac{2p-\mu}{2p}\int_{\mathbb{R}^{3}}\biggl(\frac{\mu-p}{2p-\mu}V(x)+\nabla V(x)\cdot x\biggr)|u_{\lambda}|^{p}\notag\\
		&+\int_{\mathbb{R}^{3}}\bigl(u_{\lambda}f(u_{\lambda})-\mu F(u_{\lambda})\bigr)+\biggl(\frac{r-\mu}{r}\biggr)\lambda\int_{\mathbb{R}^{3}}|u_{\lambda}|^{r}=\biggl(\frac{5\mu}{2}-3p\biggr)c_{\lambda}.\notag
	\end{align}
	By $(V_{2}),(f_{3})$,\eqref{4.9} and $\frac{2p(p+1)}{p+2}<\mu\leqslant 2p<r$, we get
	\[\frac{5\mu}{2}-3p>0,\frac{2\mu-2p-2p^{2}+p\mu}{2p}>0.\]
	Therefore, $\{\int_{\mathbb{R}^{3}}|\nabla u_{\lambda}|^{p}\},\{\lambda\int_{\mathbb{R}^{3}}|u_{\lambda}|^{r}\}$ is uniformly bounded for $\lambda\in(0,1)$.\\
	Moreover, by $\eqref{4.11}$ and $\eqref{4.12}$, we obtain
	\begin{align}
		c_{\lambda}=&\frac{r-p}{p}\int_{\mathbb{R}^{3}}|\nabla u_{\lambda}|^{p}+\frac{r-p}{p}\int_{\mathbb{R}^{3}}V(x)|u_{\lambda}|^{p}+\frac{r-2p}{2p}\phi_{u_{\lambda}}|u_{\lambda}|^{p}\notag\\
		&+\int_{\mathbb{R}^{3}}\bigl(u_{\lambda}f(u_{\lambda})-rF(u_{\lambda})\bigr)+\biggl(\frac{r-1}{r}\biggr)\lambda\int_{\mathbb{R}^{3}}|u_{\lambda}|^{r}.\notag
	\end{align}
	One sees that $\int_{\mathbb{R}^{3}}V(x)|u_{\lambda}|^{p}$ is uniformly bounded for $\lambda\in(0,1)$. Therefore, $\{u_{\lambda}\}$ is bounded in $E$ for $\lambda\in(0,1)$. Up to a
	subsequence, $u_{\lambda}\rightharpoonup u$ in $E$ for $\lambda\to0^{+}$ . By Remark \ref{remark2.2}, $u_{\lambda}\to u$ in $L^{q}(\mathbb{R}^{N})$. By the Lemma \ref{lemma2.1}, $\phi_{u_{n}}\to\phi_{u}$ in $D^{1,2}(\mathbb{R}^{3})$. And by Ekeland's variational principle and deformation lemma, we have $I'(u)=0$. Thus, $u_{\lambda}\to u$ in $E$ for $\lambda\to0^{+}$. Moreover, the fact that $u_{\lambda}\in E\backslash(P_{\varepsilon}^{+}\cup P_{\varepsilon}^{-})$ and $I_{\lambda}(u_{\lambda})=c_{\lambda}\geqslant\frac{\varepsilon^{p}}{p}$ for $\lambda\in(0,1)$ implies $u\in E\backslash(P_{\varepsilon}^{+}\cup P_{\varepsilon}^{-})$ and $I(u)\geqslant\frac{\varepsilon^{p}}{p}$. Therefore, $u$ is a sign-changing solution of equation \eqref{4.1} i.e. system \eqref{1.1}.
\end{proof}

In the following, we prove the existence of infinitely many sign-changing solutions to \eqref{2.1}. We assume that $f$ is odd. Thanks to Lemma \ref{lemma4.1}-\ref{lemma4.4}, we have seen that $P_{\varepsilon}^{+}$ is a $G-$admissible invariant set for the functional $I_{\lambda}(0<\lambda<1)$ at any level $c$.

\begin{proof}
	\textbf{(Multiplicity part)} We divide the proof into three steps.\\
	\textbf{Step1.} Referring to \cite{liu2015multiple}, we construct $\varphi_{n}$ satisfying the assumptions in Theorem \ref{theorem3.2}. For any $n\in\mathbb{N}$, we choose $\{v_{i}\}_{1}^{n}\subset C_{0}^{\infty}(\mathbb{R}^{3})\backslash \{0\}$ such that $i\ne j,\;\text{supp}(v_{i})\cap\text{supp}(v_{j})=\emptyset$. Define $\varphi_{n}\in C(B_{n},E)$ as
	\[\varphi_{n}(t)(\cdot)=R_{n}^{2}\sum_{i=1}^{n}t_{i}v_{i}(R_{n}\cdot),\qquad t=(t_{1},t_{2},\cdots,t_{n})\in B_{n},\]
	where $R_{n}>0$ is a large number independent of $\lambda$ such that $\varphi_{n}(\partial B_{n})\cap(P_{\varepsilon}^{+}\cap P_{\varepsilon}^{-})=\emptyset$, which means the condition $(2)$ of Theorem \ref{theorem3.2} is satisfied.
	
	Moreover,
	\[\sup_{u\in\varphi_{n}(\partial B_{n})}I_{\lambda}(u)<0<\inf_{u\in \Sigma}I_{\lambda}(u)=c_{\lambda}^{*}.\]
	For the perturbed problem, the condition $(3)$ of Theorem \ref{theorem3.2} is satisfied. Obviously, $\varphi_{0}(0)=0\in P_{\varepsilon}^{+}\cap P_{\varepsilon}^{-},\varphi_{n}(-t)=-\varphi_{n}(t)$ for $t\in B_{n}$. This means the condition $(1)$ of Theorem \ref{theorem3.2} is satisfied.
	
	\textbf{Step2.} For any $j\in\mathbb{N},\lambda\in(0,1)$, we define
	\[c_{j}(\lambda)=\inf_{B \in \Gamma_{j}}\sup_{u \in B \backslash W}I_{\lambda}(u),\]
	where, $W:=P_{\varepsilon}^{+}\cup P_{\varepsilon}^{-}$ and $\Gamma_{j}$ match the definition of Theorem \ref{theorem3.2}. According to Theorem \ref{theorem3.2}, for any $0<\lambda<1,j\geqslant2$,
	\[0<\inf_{u \in \Sigma}I_{\lambda}(u):=c_{\lambda}^{*}\leqslant c_{j}(\lambda)\to\infty,\qquad \text{as} \;j\to\infty.\]
	Similar to the Step2 of the existence, there exists $\{u_{\lambda,j}\}_{j\geqslant2}\subset E\backslash W$ such that $I_{\lambda}(u_{\lambda,j})=c_{j}(\lambda)$, $I_{\lambda}'(u_{\lambda,j})=0$.
	
	\textbf{Step3.} Similarly, for any fixed $j\geqslant2,\{u_{\lambda,j}\}_{\lambda\in(0,1)}$ is bounded in $E$. Up to
	subsequence, $u_{\lambda,j}\rightharpoonup u_{j}$ as $\lambda\to0^{+}$ in $E$. Observe that $c_{j}(\lambda)$ is decreasing in $\lambda$. Let $c_{j}=\lim_{\lambda\to0^{+}}c_{j}(\lambda)$. Clearly $c_{j}(\lambda)\leqslant c_{j}<\infty$ for $\lambda\in(0,1)$. There exists  $u_{\lambda,j}\to u_{j}$, such that $I'(u_{j})=0,I(u_{j})=c_{j}$ as $\lambda\to0^{+}$ for some critical points $u_{j}\in E\backslash W$. Since $c_{j}\geqslant c_{j}(\lambda)$ and $\lim_{j\to\infty}c_{j}(\lambda)=\infty,\lim_{j\to\infty}c_{j}=\infty$. Therefore, the system \eqref{1.1} has infinitely many sign-changing solutions. The proof is completed.
\end{proof}

\section*{Acknowledgement}

This research is supported by the Fundamental Research Funds for the Central Universities 2019XKQYMS91 of China.

\bibliographystyle{tfnlm}
\bibliography{interactnlmsample.bib}

\begin{thebibliography}{10}
\providecommand{\url}[1]{\normalfont{#1}}
\providecommand{\urlprefix}{Available from: }

\bibitem{liu2016infinitely}
Liu~Z, Wang~ZQ, Zhang~J. Infinitely many sign-changing solutions for the
  nonlinear schr{\"o}dinger--poisson system. Annali di Matematica Pura ed
  Applicata. 2016;\hspace{0pt}195(3):775--794.

\bibitem{benguria1981thomas}
Benguria~R, Br{\'e}zis~H, Lieb~EH. The thomas-fermi-von weizs{\"a}cker theory
  of atoms and molecules. Communications in Mathematical Physics.
  1981;\hspace{0pt}79(2):167--180.

\bibitem{lieb1997thomas}
Lieb~EH. Thomas-fermi and related theories of atoms and molecules. In: The
  stability of matter: From atoms to stars. Springer; 1997. p. 259--297.

\bibitem{lions1987solutions}
Lions~PL. Solutions of hartree-fock equations for coulomb systems.
  Communications in Mathematical Physics. 1987;\hspace{0pt}109(1):33--97.

\bibitem{markowich2012semiconductor}
Markowich~PA, Ringhofer~CA, Schmeiser~C. Semiconductor equations. Springer
  Science \& Business Media; 2012.

\bibitem{ambrosetti2008multiple}
Ambrosetti~A, Ruiz~D. Multiple bound states for the schr{\"o}dinger--poisson
  problem. Communications in Contemporary Mathematics.
  2008;\hspace{0pt}10(03):391--404.

\bibitem{benci1998eigenvalue}
Benci~V, Fortunato~D. An eigenvalue problem for the schr{\"o}dinger-maxwell
  equations. Topological Methods in Nonlinear Analysis.
  1998;\hspace{0pt}11(2):283--293.

\bibitem{benci2002solitary}
Benci~V, Fortunato~DF. Solitary waves of the nonlinear klein-gordon equation
  coupled with the maxwell equations. Reviews in Mathematical Physics.
  2002;\hspace{0pt}14(04):409--420.

\bibitem{bartsch2004superlinear}
Bartsch~T, Liu~Z. On a superlinear elliptic p-laplacian equation. Journal of
  Differential Equations. 2004;\hspace{0pt}198(1):149--175.

\bibitem{bartsch2005nodal}
Bartsch~T, Liu~Z, Weth~T. Nodal solutions of a p-laplacian equation.
  Proceedings of the London Mathematical Society.
  2005;\hspace{0pt}91(1):129--152.

\bibitem{du2021schrodinger}
Du~Y, Su~J, Wang~C. The schr{\"o}dinger--poisson system with p-laplacian.
  Applied Mathematics Letters. 2021;\hspace{0pt}120:107286.

\bibitem{du2022quasilinear}
Du~Y, Su~J, Wang~C. On a quasilinear schr{\"o}dinger-poisson system. Journal of
  Mathematical Analysis and Applications. 2022;\hspace{0pt}505(1):125446.

\bibitem{liu2015multiple}
Liu~J, Liu~X, Wang~Zq. Multiple mixed states of nodal solutions for nonlinear
  schr{\"o}dinger systems. Calculus of Variations and Partial Differential
  Equations. 2015;\hspace{0pt}52(3):565--586.

\bibitem{liu2018system}
Liu~X, Zhao~J, Liu~J. On the system of p-laplacian equations with critical
  growth. International Journal of Mathematics.
  2018;\hspace{0pt}29(02):1850008.

\bibitem{cao2012infinitely}
Cao~D, Peng~S, Yan~S. Infinitely many solutions for p-laplacian equation
  involving critical sobolev growth. Journal of Functional Analysis.
  2012;\hspace{0pt}262(6):2861--2902.

\bibitem{ambrosetti1996multiplicity}
Ambrosetti~A, Azorero~JG, Peral~I. Multiplicity results for some nonlinear
  elliptic equations. Journal of Functional Analysis.
  1996;\hspace{0pt}137(1):219--242.

\bibitem{ambrosetti1997quasilinear}
Ambrosetti~A, Azorero~JG, Peral~I. Quasilinear equations with a multiple
  bifurcation. Differential and Integral Equations.
  1997;\hspace{0pt}10(1):37--50.

\bibitem{azorero1994some}
Azorero~JG, Alonso~IP. Some results about the existence of a second positive
  solution in a quasilinear critical problem. Indiana University Mathematics
  Journal. 1994;\hspace{0pt}:941--957.

\bibitem{chabrowski1995multiple}
Chabrowski~J. On multiple solutions for the nonhomogeneous $ p $-laplacian with
  a critical sobolev exponent. Differential and Integral Equations.
  1995;\hspace{0pt}8(4):705--716.

\bibitem{garcia1991multiplicity}
Garcia~Azorero~J, Peral~Alonso~I. Multiplicity of solutions for elliptic
  problems with critical exponent or with a nonsymmetric term. Transactions of
  the American Mathematical Society. 1991;\hspace{0pt}323(2):877--895.

\bibitem{pomponio2010schrodinger}
Pomponio~A, Azzollini~A, d'Avenia~P. On the schr{\"o}dinger--maxwell equations
  under the effect of a general nonlinear term. Annales de l'Institut Henri
  Poincar{\'e} C. 2010;\hspace{0pt}27(2):779--791.

\bibitem{azzollini2008ground}
Azzollini~A, Pomponio~A. Ground state solutions for the nonlinear
  schr{\"o}dinger--maxwell equations. Journal of Mathematical Analysis and
  Applications. 2008;\hspace{0pt}345(1):90--108.

\bibitem{bartsch1995existence}
Bartsch~T, Qiang~Wang~Z. Existence and multiplicity results for some
  superlinear elliptic problems on ${R}^{N}$: Existence and multiplicity
  results. Communications in Partial Differential Equations.
  1995;\hspace{0pt}20(9-10):1725--1741.

\bibitem{d2006standing}
D'Aprile~T, Wei~J. Standing waves in the maxwell-schr{\"o}dinger equation and
  an optimal configuration problem. Calculus of Variations and Partial
  Differential Equations. 2006;\hspace{0pt}25(1):105--137.

\bibitem{d2004non}
D'Aprile~T, Mugnai~D. Non-existence results for the coupled
  klein-gordon-maxwell equations. Advanced Nonlinear Studies.
  2004;\hspace{0pt}4(3):307--322.

\bibitem{d2002non}
d'Avenia~P. Non-radially symmetric solutions of nonlinear schr{\"o}dinger
  equation coupled with maxwell equations. Advanced Nonlinear Studies.
  2002;\hspace{0pt}2(2):177--192.

\bibitem{ianni2008concentration}
Ianni~I, Vaira~G. On concentration of positive bound states for the
  schr{\"o}dinger-poisson problem with potentials. Advanced nonlinear studies.
  2008;\hspace{0pt}8(3):573--595.

\bibitem{liu2001invariant}
Liu~Z, Sun~J. Invariant sets of descending flow in critical point theory with
  applications to nonlinear differential equations. Journal of Differential
  Equations. 2001;\hspace{0pt}172(2):257--299.

\bibitem{lieb2001analysis}
Lieb~EH, Loss~M. Analysis. Vol.~14. American Mathematical Soc.; 2001.

\end{thebibliography}
\end{document}